\documentclass[3p, 11pt]{PaperBis}

\usepackage{amscd} 
\usepackage{verbatim}
\usepackage{amsxtra}
\usepackage{csquotes}
\usepackage{mathtools}

\usepackage{scrextend}
\usepackage{amsthm}
\usepackage{amsmath,etoolbox}
\usepackage{amssymb}
\usepackage{paralist}
\usepackage{graphics} 
\usepackage{epsfig} 
\usepackage{subcaption}
\usepackage{graphicx}  
\usepackage{epstopdf}
\usepackage[colorlinks=true]{hyperref}
\usepackage{centernot}
\usepackage[shortlabels]{enumitem}
\usepackage{extarrows}
\usepackage{mathrsfs}
\hypersetup{urlcolor=blue, citecolor=red}
\usepackage{microtype}
\usepackage[mathscr]{euscript}
\usepackage{float}

\allowdisplaybreaks

\newtheorem{theorem}{Theorem}[section]
\newtheorem{corollary}{Corollary}[section]
\newtheorem{lemma}{Lemma}[section]
\newtheorem{proposition}{Proposition}[section]
\newdefinition{definition}{Definition}[section]
\newdefinition{remark}{Remark}[section]
\newdefinition{example}{Example}[section]

\makeatletter
\newenvironment{sqcases}{%
  \matrix@check\sqcases\env@sqcases
}{%
  \endarray\right.%
}
\def\env@sqcases{%
  \let\@ifnextchar\new@ifnextchar
  \left[
  \def\arraystretch{1.2}%
  \array{@{}l@{\quad}l@{}}%
}
\makeatother

\makeatletter
\newcommand{\raisemath}[1]{\mathpalette{\raisem@th{#1}}}
\newcommand{\raisem@th}[3]{\raisebox{#1}{$#2#3$}}
\makeatother

\def\f{\longrightarrow}
\def\im{\Longrightarrow}

\def\N{\mathbb{N}}

\def\e{\varepsilon}
\def\l{\lambda}

\def\x{\bar{x}}

\def\u{\bar{u}}
\def\a{\alpha}
\def\b{\beta}
\def\<{\langle}
\def\>{\rangle}

\def\s{\sigma}
\def\S{\mathbb{S}}
\def\r{\bar{r}}
\def\ba{\bar{a}}

\def\G{\mathcal{G}}

\def\B{\mathcal{B}}

\def\e{\varepsilon}
\def\R{\mathbb{R}}

\def\J{\mathcal{J}}
\def\k{\bar{k}}
\def\inte{\textnormal{int}\,}
\def\clo{\textnormal{cl}\,}
\def\epi{\textnormal{epi}\,}
\def\bdry{\textnormal{bdry}\,}
\def\dom{\textnormal{dom}\,}
\def\conv{\textnormal{conv}\,}

\def\gk{\gamma_k}
\def\gkn{\gamma_{k_n}}
\def\Vphi{\varPhi}

\def\CO{\mathcal{C}}
\def\c{\mathsf{c}}
\def\-{\textnormal{-}}
\def\I{\mathcal{I}}
\newcommand*{\tran}{^{\mkern-1.5mu\mathsf{T}}}

\DeclareMathOperator{\sign}{sign}
\def\sp{\hspace{0.015cm}}
\def\bp{\hspace{-0.08cm}}
\def\bbp{\hspace{-0.04cm}}

\def\bbbp{\hspace{-0.02cm}}

\def\c{\bbp c}
\def\TV{_{\textnormal{\tiny{T.\sp V.}}}}

\def\sstar{\raisebox{0.15ex}{\scaleto{\mathscr{{\star}}}{3.5pt}}}

\def\ab{\scaleto{\sp\alpha,\beta}{4.5pt}}

\def\0gk{\scaleto{0,\gk}{3.5pt}}
\def\G{\scaleto{G}{4.3pt}}
\def\hzeta{\tilde{\zeta}}

\def\tallqed{\smash{\scalebox{.62}[0.96]{$\square$}}}

\makeatletter

\binoppenalty=\maxdimen 
\relpenalty=\maxdimen 

\usepackage[noabbrev]{cleveref}

\usepackage{scalerel}

\begin{document}

\begin{frontmatter}

\title{Pontryagin-type maximum principle for a controlled sweeping process with nonsmooth and unbounded sweeping set}


\author{Chadi Nour\corref{}}
\address{Department of Computer Science and Mathematics, Lebanese American University, Byblos Campus, P.O. Box 36, Byblos, Lebanon}
\ead{cnour@lau.edu.lb}

\author{Vera Zeidan\corref{cor}}
\address{Department of Mathematics, Michigan State University, East Lansing, MI 48824-1027, USA}
\ead{zeidan@msu.edu}
\cortext[cor]{Corresponding author}

\begin{abstract} 
 The Pontryagin-type maximum principle derived in \cite{VCpaper} for optimal control problems involving sweeping processes  is generalized to the case where the sweeping set $C$ is nonsmooth and not necessarily bounded, namely, $C$ is the intersection of a finite number of zero-sublevel sets of  smooth functions. 
\end{abstract} 

\end{frontmatter}


\section{Introduction}\label{intro} 

{\it Sweeping processes} are dynamical systems involving the {\it normal cone} to a set called {\it sweeping set}. These systems were introduced in the papers \cite{moreau1,moreau2,moreau3} by J.J. Moreau in the context of plasticity and friction theory.  Different modifications of this model have since appeared in many applications such as  hysteresis, ferromagnetism, electric circuits, phase transitions, crowd motion problems, economics, etc. (see \cite{outrata} and its references).
This naturally motivated launching the subject  of optimal control over sweeping processes and  deriving {\it necessary optimality conditions}  phrased in terms of {\it Euler-Lagrange equation} or {\it Pontryagin-type maximum principle}. These results are mainly established using two different approaches, namely {\it discrete} approximations (see \cite{ccmn,ccmnbis,cmo0,cmo,cmo2, chhm2, chhm, cmn0}), and {\it continuous} approximations (see \cite{brokate,pinho,pinhonew, VCpaper, verachadi}). Another new method is recently introduced in \cite{palladino}, where, instead of approximating the sweeping process,  the authors formulated   {\it standard} optimal control problems having {\it auxiliary} controls and   the sweeping set is considered as {\it explicit} state constraints. These problems admit the same   optimal solution as that of the original problem.  As mentioned by the authors in \cite[Remark 2.13]{palladino}, the Pontryagin-type necessary conditions derived therein have  atypical nondegeneracy condition that requires  further analysis.  

Via an innovative {\it exponential penalization technique} that {\it approximates} the normal cone, 
the authors in  \cite{pinho,pinhoEr} derived a {\it smooth} Pontryagin-type maximum principle for {\it global} minimizers of a controlled sweeping process, in which the sweeping set $C$  is  {\it compact}  and is  the zero-sublevel set of a $\CO^2$-convex function $\psi$. The surprising  feature of this continuous-time approximation technique resides in the fact that, despite that the original sweeping process inherently has a state constraint, namely, $x(t)\in C$,  the presence of the penalty term in the formulation of the approximating systems  disposes of this state constraint,  since the set $C$ turns out  to be invariant for this carefully designed approximating systems. 
This special property of the  exponential penalization technique introduced in \cite{pinho}, has demonstrated to be instrumental  in \cite{pinhonum,verachadinum}, where  successful  numerical algorithms have been developed to efficiently compute  {\it good} approximations of  solutions for optimal control problems over sweeping processes via solutions of {\it standard} optimal control problems in which {\it no} state-constraint is  imposed.
 
The same exponential penalization technique is used in \cite{VCpaper, verachadi} to generalize the Pontryagin-type maximum principle of \cite{pinho} in several directions. This includes allowing the compact sweeping set $C$  to be the zero-sublevel set of a $\CO^{1,1}$-function  that is {\it not} necessarily convex, a  final-state endpoint constraint set to be present, the cost function to depend on both endpoints of the state, strong local minimizers to be considered, and new subdifferentials that are strictly smaller than Clarke and Mordukhovich subdifferentials to be employed.

The goal of this paper is to extend the Pontryagin-type maximum principle of \cite{VCpaper} to the case in which the sweeping set is the {\it intersection} of a finite number of zero-sublevel sets of  $\CO^{1,1}$-functions $\psi_i$, $i=1,\dots, r$, and hereby, we allow  $C$ to be nonsmooth. Moreover, the compactness assumption of $C$  can  now be replaced by a weaker condition permitting $C$ to be {\it unbounded}, and hence, our results here extend the  results in \cite{pinho,VCpaper, verachadi},  not only to the case where the sweeping set is  {\it nonsmooth},  but also {\it unbounded}, such as, polyhedral, star-shaped,  or a set with compact boundaries. 

The  exponential penalization technique  for $r > 1$ is far more complex than that for $r=1$, which was  treated in  \cite{pinho,VCpaper, verachadi}. In fact, if  we attempt here to emulate the case when $r=1$ by considering  $C$  as the zero-sublevel  set  of the single function  $\psi:= \hbox{max}\{\psi_i: i=1,\dots, r\}$,  then, when $r>1$ and  $\psi_i\in \CO^{1,1}$, for all $i\in\{1,\dots,r\}$,  $\psi$ is only guaranteed to be  Lipschitz.
 Hence, unlike the case when $r=1$, the normal cone to $C, N_C$,   is now given in terms of the subgradient  $\partial\psi$, as opposed to $\nabla \psi$,  and this set-valued map  is not in general Lipschitz. This  would lead to approximating the normal cone in the sweeping process by an {\it exponential penalty} term invoking  $\partial\psi$, which produces  a differential inclusion to which none of the standard results would apply. Thus,  this attempt would not be successful. 
 On the other hand, if we express the normal cone to $C$  at a boundary point $x$ as the cone combination of the  corresponding $\nabla\psi_i(x)$, the approximation of $N_C(x)$ in the sweeping process  via the {exponential penalty} technique leads, as in the case $r=1$, to an ordinary control system, which now involves $r$ {\it exponential penalty} terms, instead of only one term.  The presence of $r$- penalty terms causes a major obstacle when attempting to prove that $C$ is invariant for the resulting control system, unless one imposes a very {\it restrictive} assumption  such as: $\<\nabla\psi_i(x), \nabla\psi_j(x)\> \ge 0$  in a band around the boundary of $C$, which excludes  sets having acute angles, like triangles, etc.  During the final writing of this paper, we became aware of the manuscript \cite{pinho22} posted on arXiv in which this latter condition is imposed.   
 
 Therefore, it is clear that in order to deal with the intricacy of the situation for $r>1$ and to circumvent {\it both} major obstacles mentioned above, a new idea is required.  Our approach in this paper  is to  approximate  the {\it nonsmooth} max-function $\psi$, defining the sweeping set $C$, by a sequence  $(\psi_{\gk})_k$ of  $\CO^{1,1}$-functions, constructed by applying the \enquote{epi-multiple log-exponential} operator $\big(\frac{1}{\gk}\sstar\log\bbp\exp\big)$ to the \enquote{vecmax} function  $\psi$ (see \cite[Example 1.30]{rockwet}).  This allows us to approximate the nonsmooth set $C$ by a sequence of smooth sets $C^{\gk}$,  the zero-sublevel sets of $\psi_{\gk}$.  We employ the exponential penalization  technique for  $\psi_{\gk}$ which produces  the {\it same} approximating control system that has  $r$-penalty terms involving  $\nabla\psi_i$, $i=1,\cdots,r$. For this system, it turns out  that   $C^{\gk}$, instead of $C$, is actually {\it invariant}  under {\it no} extra assumption. 
 
 This result allowed us to move towards proving a Pontryagin-type maximum principle for $r>1$ and for also nonsmooth data. However, since the generalized  Hessian  of the function $\psi_{\gk}$ is {\it not} bounded near the given optimal state $\bar{x}$, another complexity  surfaces when showing the sequence of adjoint variables $p_{\gk}$ for the approximating problems has uniform bounded  variation. This is resolved by reverting to the  form of the approximating dynamic in terms of $(\nabla\psi_i)^{r}_{i=1}$ and by assuming a {\it local} condition at the values $\bar{x}(t)$ that invokes  the gradients of the corresponding active constraints and that is automatically satisfied when $r=1$. Note that, unlike \cite{pinho22},  we do {\it not} impose  any assumption like  $\nabla\psi_i=0$  on the complement in $C$ of  a band around  the boundary of $C$. The absence of this assumption causes the proof to be more challenging.
 
The layout of the paper is as follows. In  the next section, we first display our basic notations, then, we state our optimal control  problem $(P)$ over a sweeping process, and we list our hypotheses. In Section \ref{main}, we present our main result Theorem \ref{thm1mpcomp}, in which we derive nonsmooth Pontryagin-type maximum principle for {\it local} minimizers of $(P)$.  Section \ref{prepa} consists of preparatory results  that  are instrumental for the proof of Theorem \ref{thm1mpcomp} established in Section \ref{proofthm1mp}.  To promote continuous flow in the presentation of the results,  some of the proofs are provided in  the appendix (Section \ref{auxresults}), where we also show some auxiliary results that will be used in several places of the paper. An example  illustrating the utility of  Theorem \ref{thm1mpcomp}  is provided in Section \ref{auxresults}.

\section{Preliminaries} \label{Preli} 

\subsection{Notations}\label{Notations} 

We begin this subsection by briefly presenting the basic notations used in this paper. We use $\|\cdot\|$ and $\<\cdot,\cdot\>$ for the Euclidean norm and the usual inner product, respectively. The open unit ball, the closed unit ball, and the unit sphere are denoted by $B$, $\bar{B}$, and $\S$, respectively. For $x\in\R^n$ and $\rho\geq0$, the open ball, the closed balls, and the sphere of radius $\rho$ centered at $x$ are written as $B_{\rho}(x)$, $\bar{B}_{\rho}(x)$, and $\S_{\rho}(x)$, respectively. For a set $A \subset \R^n$, $\inte A$, $\bdry A$, $\clo A$, $\conv A$, and  $A^c$ designate  the interior, the boundary, the closure, the convex hull, and the complement of $A$, respectively. The set $A\subset\R^n$ is said to be star-shaped if there exists $a_o\in A$, called center of $A$, such that the closed interval $[a_o,x]\subset A$ for all $x\in A$.  When $A\subset\R^n$ is closed and convex, we denote by $h_A$ the support function of $A$. For $f\colon\R^n\f\R\cup\{\infty\}$ an extended-real-valued function, $\dom f$ is the effective domain of $f$ and $\epi f$ is its epigraph. The Lebesgue space of $p$-integrable functions $f\colon [a,b]\f\R^n$ is denoted by $L^p([a,b];\R^n)$. The norms in $L^p([a,b];\R^n)$ and $L^{\infty}([a,b];\R^n)$ (or $C([a,b];\R^n)$) are written as $\|\cdot\|_p$ and $\|\cdot\|_{\infty}$, respectively. The space  $W^{1,1}([a,b];\R^n)$ denotes the set of all absolutely continuous functions $x\colon [a,b] \to \R^n$. The set of all functions $f\colon [a,b]\f\R^n$ of bounded variations is denoted by $BV([a,b];\R^n))$. The space $C^*([a,b]; \R)$ denotes the dual of $C([a,b];\R)$ equipped with the supremum norm. We denote by   $\|\cdot\|\TV$ the induced norm on $C^*([a,b]; \R)$. By Riesz representation theorem, each element in  $C^*([a,b]; \R)$ can be interpreted as an element in the space of finite signed Radon measures on $[a,b]$ equipped with the weak* topology. For the set of $m\times n$-matrix functions  on $[a,b]$, we use $\mathscr{M}_{m\times n}([a, b])$. For $A\subset\R^d$ compact, $C(A;\R^n)$ denotes to the set of  continuous functions from  $A$ to $\R^n$. 

In what follows, we display some notations from {\it nonsmooth analysis}. For standard references, see the monographs \cite{clarkeold,clsw,mordubook,rockwet}. Let $A$ be a nonempty and closed subset of $\R^n$, and let $a\in A$. The {\it proximal}, the {\it Mordukhovich} ({also known as \it limiting}), and the {\it Clarke normal} cones to $A$ at $a$ are denoted by $N^P_A(a)$, $N_A^L(a)$, and $N_A(a)$, respectively. For the {\it Clarke tangent} cone to $A$ at $a$, we use $T_A(a)$. Given a {\it lower semicontinuous} function $f\colon\R^n\f\R\cup\{\infty\}$, and  $x \in \dom f$, the \textit{proximal}, the {\it Mordukhovich} (or \textit{limiting}), and the \textit{Clarke subdifferential}  of $f$ at  $x$ are denoted by  $\partial^P f (x)$, $\partial^L f(x)$, and $\partial f(x)$, respectively. Note that if  $x\in\inte (\dom{f})$ and $f$ is Lipschitz near $x$, \cite[Theorem 2.5.1]{clarkeold} yields that the Clarke subdifferential of $f$ at  $x$ coincides  with the {\it Clarke generalized gradient} of ${f}$ at $x$, also denoted here by $\partial {f}(x)$. If ${f}$ is $\CO^{1,1}$ near $x\in \inte(\dom f)$,  $\partial^2f(x)$ denotes the {\it Clarke generalized Hessian} of ${f}$ at $x$. For  $g\colon\R^n \f \R^n$ Lipschitz near $x\in\R^n$, $\partial{g}(x)$ denotes the {\it Clarke generalized Jacobian} of $g$ at $x$.

\subsection{Statement of the problem $(P)$, and assumptions}\label{Assump} 
In this paper, we consider the following fixed time Mayer problem 
$$\begin{array}{l} (P)\colon\; \hbox{Minimize}\;g(x(0),x(T))\vspace{0.1cm}\\ \hspace{0.9cm} \hbox{over}\;(x,u)\;\hbox{such that}\;u(\cdot)\in \mathscr{U}\;\hbox{and}\\[2pt] \hspace{0.9cm}  \begin{cases} (D)  \begin{sqcases}\dot{x}(t)\in f(t,x(t),u(t))-\partial \varphi(x(t)),\;\;\hbox{a.e.}\;t\in[0,T],\\x(0)\in C_0\subset \dom \varphi, \end{sqcases}\vspace{0.1cm}\\ x(T)\in C_T, \end{cases}
    \end{array}$$
    where $T>0$ is fixed,  $g\colon\R^n\times\R^n\f\R\cup\{\infty\}$, $f\colon[0,T]\times\R^n\times \R^m\f\R^n$, $\varphi\colon\R^n\f\R\cup\{\infty\}$ is lower semicontinuous,  $\partial\varphi$ stands for the Clarke subdifferential  of $\varphi$, $C:=\dom\varphi$ is the intersection of the zero-sublevel sets of a finite sequence of functions $\psi_i\colon\R^n\f\R$, $i=1,\dots,r$, $C_0\subset C$ is the initial constraint set, $C_T\subset\R^n$ is the terminal constraint set, and, for $U\colon[0,T]\rightrightarrows\R^{m}$ a multifunction, the set of control functions $\mathscr{U}$ is defined as
\begin{equation*}\label{setu}
\mathscr{U}:=\left\{ u\colon[0,T]\f\R^m \; \;  \hbox{is  measurable and}\;\; u(t) \in U(t) \;\;\hbox{a.e.}\; t\in [0,T]\right\}. 
\end{equation*}
 A pair $(x,u)$ is {\it admissible} for $(P)$ when $x\colon[0,T]\f\R^n$ is absolutely continuous,   $u\in \mathscr{U}$, and $(x,u)$  satisfies the {\it perturbed sweeping process} $(D)$, called the dynamic of $(P)$. An admissible pair $(\x,\u)$ is said to be a {\it strong local minimizer} if there exists $\delta >0$ such that $$g(\x(T))\le g(x(T)), \;\; \forall (x,u) \;\;\hbox{admissible for}\;\; (P)
\;\hbox{and satisfying}\;\|x- \x\|_{\infty}\leq\delta.	$$
Note that the admissibility  of  a pair $(x,u)$ yields from $(D)$ that $x(t)\in C$, $\forall t\in[0,T]$.

Now let $(\x,\u)$ be a strong local minimizer of $(P)$. Next, we present our assumptions on the data of $(P)$ needed to derive our Pontryagin-type maximum principle for $(\x,\u)$. Note that not all these assumptions are needed for our intermediate and auxiliary results. Moreover, a global version of (A1) is introduced later in Section \ref{prepa}. Define the set  $ \mathbb{U}:=\bigcup_{t\in[0,T]} U(t)$.
\begin{enumerate}[label=\textbf{A\arabic*}:]
\item For fixed $(x,u)\in C_{\x}\times  \mathbb{U } $, $f(\cdot, x,u)$ is Lebesgue-measurable; and there exists $M_{\ell}>0$ such that, for $t \in [0,T]$ a.e., we have: $(x, u)\mapsto f (t, x, u)$ is continuous on $(C \cap \bar{B}_{\delta}(\x(t)))\times U(t)$; for all $u\in U(t)$, $x\mapsto f (t, x, u)$ is $M_{\ell}$-Lipschitz on $C \cap \bar{B}_{\delta}(\x(t))$; and  $\|f(t,x,u)\| \leq M_{\ell}$  for all $(x,u)\in (C \cap \bar{B}_{\delta}(\x(t))) \times U(t).$ 
\item The set $C:=\dom \varphi\not=\emptyset$ is given by 
\begin{equation} \label{Cdef} C:=\bigcap_{i=1}^{r}C_i,\;\hbox{where}\;C_i:=\{x\in\R^n : \psi_i(x)\leq0\}\;\hbox{and}\end{equation} 
$(\psi_i)_{1\leq i\leq r}$ is a family of functions $\psi_i\colon\R^n\f\R$.
\begin{enumerate}[label=\textbf{A2.\arabic*}:]\vspace{0.1cm}
\item There exists $\rho>0$ such that for each $i$, the function $\psi_i$ is $\CO^{1,1}$ on $ C+\rho{B}$.
\item There is a constant $\eta>0$ such that $$ \left\|\sum_{i\in\I^0_c}\lambda_i\nabla\psi_i(c)\right\|>2\eta, \;\;  \forall c\in \{x\in C: \I^0_x\not=\emptyset\},$$ where $\I^0_x:=\{i\in\{1,\dots,r\} : \psi_i(x)=0\}$ and $(\l_i)_{i\in\I^0_c}$ is any sequence of nonnegative numbers satisfying $\sum_{i\in\I^0_c}\l_i=1$.
\item There exists $b\in(0,1)$ such that  $$\sum_{\substack{i\in\I^0_{\x(t)}\\ i\not=j}}|\<\nabla \psi_i(\x(t)),\nabla\psi_j(\x(t))\>|\leq b\|\nabla\psi_j(\x(t))\|^2,\;\;\forall t\in I^0(\x)\;\,\hbox{and}\;\,\forall j\in\I^0_{\x(t)},$$ where $I^0(\x):=\{t\in[0,T] : \x(t)\in\bdry C\}=\{t\in [0,T] : \I^0_{\x(t)}\not=\emptyset\}$.
\item There exist $y_o\in\R^n$ and $R_o>0$ such that:
\begin{enumerate}[leftmargin=1.1cm, label=\textbf{{({\it \roman*}\sp\sp)}}:]
\item For all $c\in(\bdry C)\cap \S_{R_o}(y_o)$, we have $(y_o-c)\not\in N_C(c)$.
\item $\x(t)\in B_{R_o}(y_o)$ for all $t\in [0,T]$.
\end{enumerate}
\item The set $C$ has a connected interior.\footnote{This assumption is only imposed to obtain the required quasiconvexity of $C$ for the extension function $\Vphi$ of $\varphi$, see Proposition \ref{prop1}$(ii)$. Thus, when such an extension is readily available, as is the case  when $\varphi$ is the {\it indicator} function of $C$, assumption (A2.5) would be superfluous. For more information about this assumption, see \cite[Remark 3.2]{VCpaper}.}
\end{enumerate}
\item The function $\varphi$ is globally Lipschitz on $C$ and $\CO^{1}$ on $\inte C$, and the function $\nabla\varphi$ is globally Lipchitz on $\inte C$.
\item The following assumptions on $C_0$, $C_T$ and $U$ hold: \begin{enumerate}[label=\textbf{A4.\arabic*}:]
\item The set $C_0\subset C$ is nonempty and closed.
\item The set $C_T\subset\R^n$ is nonempty and closed. 
\item The graph of $U(\cdot)$ is a $\mathscr{L}\times\mathscr{B}$ measurable set, and,  for $t\in [0,T],$ $U(t)$ is closed, and bounded uniformly in $t$.
\end{enumerate}
\item There exists  $\tilde{\rho}>0$ such that $g$ is $L_g$-Lipschitz on $\tilde{C}_0(\delta)\times\tilde{C}_T(\delta)$, where  \begin{equation*}\label{c0deltactdelta}  \tilde{C}_0(\delta):=\big[\big(C_0\cap \bar{B}_{\delta}(\x(0))\big)+\tilde{\rho}\bar{B}\big]\cap C\,\;\hbox{and}\;\, \tilde{C}_T(\delta):=\big[\big(C_T\cap \bar{B}_{\delta}(\x(T))\big)+\tilde{\rho}\bar{B}\big]\cap C.\end{equation*}
\item The set $f(t,x,U(t))$ is convex for all $x\in C \cap \bar{B}_{\delta}(\x(t))$ and  $t\in [0,T]$ \textnormal{a.e.}\footnote{When  $C_T=\R^n$,  this assumption is {\it not} required for the main result, i.e., Theorem \ref{thm1mpcomp}.}
\end{enumerate}

\begin{remark} \label{assumptionH}  Assumption (A2.4) is {\it weaker} than the compactness of $C$. In fact,  by Lemma \ref{conditionH}, assumption (A2.4)$(i)$ is satisfied by an {\it arbitrary large} $R_o$  when the set $C$ has a compact  boundary  or  is star-shaped (which includes {\it unbounded} convex and polyhedral sets). Hence,  taking  $R_o$ large enough, we guarantee that also  (A2.4)$(ii)$ is valid by those sets.

\end{remark}

\begin{remark}\label{epsilon}  One can easily prove that (A2.1)-(A2.2) imply that, for each $x\in \bdry C$, the family of vectors $\{\nabla\psi_i(x)\}_{i\in\I^0_x}$ is positive-linearly independent. On the other hand, (A2.1) and (A2.3) imply that for each $t\in I^0(\x)$, the family of vectors $\{\nabla\psi_i(\x(t))\}_{i\in\I^0_{\x(t)}}$ is linearly independent.
\end{remark}

\section{Main result} \label{main} In this section,  we display  our  Pontryagin-type maximum principle for strong local minimizers of $(P)$ whose proof will be given to Section \ref{proofthm1mp}.  We employ in its statement   the following {\it nonstandard} notions of subdifferentials, which are {\it strictly smaller} than their standard counterparts:
\begin{itemize}
\item $\partial_\ell\varphi$ and $\partial^2_\ell\varphi$ are the {\it extended  Clarke generalized gradient} and the {\it extended  Clarke generalized Hessian} of $\varphi$ defined on $C$, respectively (see \cite [Equation (9)-(10)]{VCpaper}). Note that if $\partial_\ell\varphi(x)$ is a singleton, then we use the notation $\nabla_{\bp\ell}$ instead of  $\partial_\ell$.
\item $\partial^{\sp x}_\ell f (t,\cdot,u)$ is the {\it extended  Clarke generalized Jacobian} of $f(t,\cdot,u)$ defined  $C \cap \bar{B}_{\delta}(\x(t))$ (see \cite [Equation (12)]{VCpaper}).
\item $\partial^2_\ell\psi$ is the {\it Clarke generalized Hessian relative} to $\inte C$ of $\psi$ (see \cite [Equation (11)]{VCpaper}).
\item $\partial^L_\ell g$ is the {\it limiting subdifferential} of $g$ {\it relative} to  $\inte \big(\tilde{C}_0(\delta)\times \tilde{C}_1(\delta)\big)$ (see \cite [Equation (8)]{VCpaper}).
\end{itemize}
For given $x(\cdot)\in W^{1,1}( [0,T]; C)$, we define the sets
\begin{equation}\label{Ibtaui}
I_i^{\-}(x):=\{t\in [0,T]:  x(t)\in \inte C_i\}\;\;\hbox{and}\;\;I_i^{0}(x):= [0,T]\setminus I_i^{\-}(x), \;\; \forall i=1,\dots, r.
\end{equation} 
\begin{theorem}[Generalized Maximum Principle for $(P)$] \label{thm1mpcomp} Let $(\x,\u)$ be a strong local minimizer for $(P),$ and assume that {\textnormal{(A1)-(A6)}} hold. Then, there exist an adjoint vector $p \in BV([0,T];\R^n)$, a finite sequence of finite signed Radon measures $(\nu^i)_{i=1}^{r}$ on $[0,T]$ with $\nu_i$ supported in $I_i^{0}(\x)$ for $i=1,\dots,r$, a finite sequence of nonnegative functions $(\xi^i)_{i=1}^{r}\in L^\infty([0,T];\R^+)$ with $\xi_i$ supported in $I_i^{0}(\x)$ for $i=1,\dots,r$, $L^{\infty}$-functions $\zeta(\cdot)$, $\theta(\cdot)$ and a finite sequence $(\vartheta_i(\cdot))_{i=1}^{r}$ in $\mathscr{M}_{n\times n}([0,T])$, and $\l\geq 0$ satisfying the following\sp$:$
\begin{enumerate}[$(i)$]
\item {\bf(The primal-dual admissible equation)}
\begin{enumerate}[$(a)$]
\item $\dot{\x}(t)= f(t, \x(t),\u(t))- \nabla_{\bp\ell}\sp\sp\varphi(\x(t))-\sum_{i=1}^{r}\xi^i(t) \nabla\psi_i(\x(t)),\;\; \forall\sp t\in [0, T]\; \textnormal{a.e.},$
\item  $\psi_i(\x(t))\le 0,\;\; \forall\sp t\in [0,T]$ and $\,\forall i\in \{1,\dots,r\};$
\end{enumerate}
\item {\bf (The nontriviality condition)} $$\displaystyle \|p(T)\|+ \l=1;$$ 
\item {\bf (The adjoint equation)}  $\forall\sp t\in [0,T]\; \textnormal{a.e.},$ 
$$(\zeta(t),\theta(t))\in \;\partial^{\sp x}_\ell f(t,\x(t),\u(t))\times \partial^2_\ell\varphi(\x(t)),$$ 
$$\vartheta_i(t)\in \partial^2_\ell\psi_i(\x(t))\;\,\hbox{for}\;\,i=1,\dots,r,\vspace{0.1cm}$$
 and, for any $z\in C([0,T];\R^n),$ we have\vspace{0.1cm} \begin{eqnarray*}\int_{[0,T]}\<z(t),dp(t)\>&= &\int_0^T \left\<z(t),\left(\theta(t)-\zeta(t)\tran\right) p(t)\right\>\sp dt \\[3pt]&+& \sum_{i=1}^{r} \int_0^T \xi^i(t)\left\<z(t),\vartheta_i(t)p(t)\right\>\sp dt\\ &+ &  \sum_{i=1}^{r} \int_{[0,T]} \<z(t),\nabla \psi_i(\x(t))\>\sp d\nu^i(t),\end{eqnarray*}
where the  left integral is the Riemann–Stieltjes integral of $z$ with respect to $p;$
\item {\bf (The complementary slackness conditions)} For $i=1,\dots,r,$ we have$\sp:$
\begin{enumerate}[$(a)$] \item $\xi^i(t)=0,\;\;\forall\sp t\in I_i^{\-}(\x),$
\item $\xi^i(t)\<\nabla\psi_i(\x(t)),p(t)\>=0,\;\;\forall\sp t\in [0, T]\; \textnormal{a.e.};$
\end{enumerate}
\item {\bf (The transversality condition)} $$ (p(0),-p(T))\in \l\partial_\ell^L g(\x(0),\x(T))+ \big[N_{C_0}^L(\x(0))\times N^L_{C_T}(\x(T))\big];$$
\item {\bf (The maximization condition)} $$\max_{u\in U}\left\<f(t,\x(t),u), p(t)\right\>\;\hbox{is attained at}\;\,\u(t)\;\,\hbox{for}\;\,\textnormal{a.e.}\;t\in [0,T].$$
\end{enumerate}
Furthermore, if $C_T=\R^n$ then, $\l=1$ and the assumption \textnormal{(A6)} is discarded.
\end{theorem}

\section{Preparatory results} \label{prepa} In this section, we present  preparatory results that are fundamental for the proof of Theorem \ref{thm1mpcomp}. In addition,  an existence theorem for  an optimal solution of $(P)$ is provided in Proposition \ref{existencesolP}. 

We assume throughout this section that $C$ is {\it compact}.  The last step in  the proof of Theorem  \ref{thm1mpcomp} (Section  \ref{proofthm1mp}) provides a technique that permits replacing  the compactness of $C$   by the weaker assumption (A2.4). 

 We denote by $\bar{M}_\psi$ a {\it common upper bound} on $C$ of the finite sequence $(\|\nabla\psi_i(\cdot)\|)_{i=1}^{r}$, and by $2M_\psi$ a {\it common Lipschitz constant}  of the finite family $\{\nabla \psi_{i}\}_{i=1}^{r}$  over the compact set  $C+\frac{\rho}{2}\bar{B}$ such that $\bar{M}_\psi\geq2\eta$ and  $M_\psi\geq \frac{4\eta}{\rho}$. 

 As for the case $r=1$ in \cite{VCpaper},   (A2.2) and the compactness of  $C$ imply  the existence of $\e>0$ such that for each $i\in\{1,\dots,r\}$ we have \begin{equation} \label{pisieps} \big[x\in C\;\hbox{and}\; \|\nabla\psi_i(x)\|\leq \eta\big]\implies \psi_i(x)<-\e.\end{equation}

\subsection{Properties of $C$, extension of $\varphi$, and new notations} We give in the following proposition some important consequences of our assumptions on the set $C$ and the function $\varphi$.

\begin{proposition}[Properties of $C$ $-$ Extension of $\varphi$] \label{prop1} Under \textnormal{(A2.1)-(A2.2)}, we have the following$\sp:$
\begin{enumerate}[$(i)$]
\item  $C$ is amenable,\textnormal{\footnote{See \cite[Chapter 10]{rockwet}.}} $\frac{\eta}{M_\psi}$-prox-regular,\textnormal{\footnote{For $\rho > 0$, the closed set $A\subset\R^n$ is said to be $\rho$-\textit{prox-regular}  if for all $a\in A$ and for all unit vector $\zeta\in N_A^P(a)$, we have  $\langle \zeta,x-a\rangle \leq \frac{1}{2\rho} \|x-a\|^2$ for all $x\in A.$ This latter inequality is known as the {\it proximal normal inequality}. For more information about prox-regularity, see \cite{prt} and references therein.}} epi-Lipschitz with $C=\clo(\inte C)$,\textnormal{\footnote{A closed set $A\subset\R^n$ is said to be {\it epi-Lipschitz} if for all $a\in A$, the Clarke normal cone of $A$ at $a$ is {\it pointed}, that is, $N_A(a)\cap-N_A(a)=\{0\}$. For more information about this property, see \cite{clarkeold,clsw,rockwet}.}} and, for all $x\in\bdry C$ we have 
\begin{equation} \label{proxformulaprox} N_C(x)=N^P_C(x)=N^L_C(x)=\Bigg\{\sum_{i\in\I_x}\lambda_i\nabla\psi_i(x) : \l_i\geq 0\Bigg\}\not=\{0\}.\vspace{-0.2cm}\end{equation}  Moreover, we have$\sp:$
$$C=\bigcap_{i=1}^{r}\inte C_i=\bigcap_{i=1}^{r}\{x\in\R^n : \psi_i(x)<0\}\not=\emptyset\;\;\hbox{and}\;\;\bdry C=C\cap \bigg(\bigcup_{i=1}^{r} \bdry C_i\bigg)\not=\emptyset.$$
\item If also \textnormal{(A2.5)} holds, then $\inte C$ is quasiconvex.\textnormal{\footnote{A set $A\subset\R^n$ is {\it quasiconvex} (\cite{brudnyi}) if there exists $c\geq 0$ such that any two points $a,\,b$ in $A$ can be joined by a polygonal line $\gamma$ in $A$ satisfying $l(\gamma) \leq c\sp \|a-b\|,$ where $l(\gamma)$ denotes the length of $\gamma$. }} Furthermore,  if in addition \textnormal{(A3)} holds, then there exists a function $\Vphi\in \CO^{1,1}(\R^n)$ such that\sp$:$\begin{itemize}
 \item $\Vphi$ is bounded on $\R^n$ and equals to $\varphi$ on $C$.
 \item $\Vphi$ and $\nabla\Vphi$ are globally Lipschitz on $\R^n$.
 \item For all $x\in C$ we have \begin{equation} \label{eq0}\partial \varphi(x)=\{\nabla \Vphi(x)\} + N_C(x).\end{equation}	
 \end{itemize}
\end{enumerate}
\end{proposition}

\begin{proof} ($i$): The amenability of $C$ follows from  \cite[Example 10.24 (d)]{rockwet}. For the prox-regularity of $C$ and the formula \eqref{proxformulaprox}, use \cite[Theorem 9.1]{adly} and \cite[Corollary 4.15]{prox}. Now from (A2.2) and \eqref{proxformulaprox},  we conclude that $N_C(x)$ is pointed for all $x\in\bdry C$, and hence $C$ is epi-Lipschitz, and thus,  $C=\clo (\inte C)$. Hence, as  $C\not=\emptyset$ and compact, then,  $\inte C\not=\emptyset$ and $\bdry C\not=\emptyset$. The two formulas of $\inte C$ and $\bdry C$ follow directly from $C$ being defined by \eqref{Cdef} and from \cite[Lemma 3.3]{VCpaper}.  \vspace{0.2cm}\\
($ii$): See the proof of \cite[Proposition 3.1$(iii)$]{verachadi}. \end{proof}

\begin{remark}\label{newD} The extension $\Vphi$ of $\varphi$ obtained in Proposition \ref{prop1}$(ii)$, satisfies for some $K>0$,
 $$|\Vphi(x)|\leq K,\;\;\|\nabla\Vphi(x)\|\leq K,\;\,\hbox{and}\;\,\|\nabla\Vphi(x)-\nabla\Vphi(y)\|\leq K\|x-y\|,\;\; \forall x,\,y\in\R^n.$$ 
Equation (\ref{eq0}) implies that the sweeping process $(D)$ can now be rephrased in terms of the normal cone to $C$ and the extension $\Vphi$ of $\varphi$ as the following 
$$(D) \begin{sqcases}\dot{x}(t)\in f(t,x(t),u(t))-\nabla\Vphi(x(t)) -N_C(x(t)),\;\;\hbox{a.e.}\;t\in[0,T],\\x(0)\in C_0\subset C.	 \end{sqcases}\\$$
\end{remark}
The reformulation of $(D)$ in Remark \ref{newD} motivates defining the function $f_{\Vphi}$ from $[0,T]\times\R^n\times\R^m$ to $\R^n$ by 
\begin{equation*} \label{eqVera4} f_{\Vphi}(t,x,u):=f(t,x,u)-\nabla\Vphi(x),\;\;\;\forall (t,x,u)\in [0,T]\times\R^n\times \R^m.\vspace{0.1cm}\end{equation*}
Clearly (A1)$_{\G}$, (A2.1)-(A2.2), (A2.5), and (A3) imply that the following holds true for $f_{\Vphi}$:
\begin{enumerate}[leftmargin=1.24cm, label=\textbf{(A\arabic*)$_\Vphi$}:]
\item For fixed $(x,u)\in C\times  \mathbb{U } $, $ f_\Vphi(\cdot, x, u)$ is Lebesgue-measurable; and for a.e $t \in [0,T]$ we have: $(x, u)\mapsto  f_\Vphi(t, x, u)$ is continuous on $C\times U(t)$;  for all $u\in U(t)$, $x\mapsto f_\Vphi(t,x,u)$ is $\bar{M}$-Lipschitz on $C$; and $\|f_\Vphi(t,x,u)\| \leq \bar{M}$  for all $(x,u)\in C \times U(t),$ where $\bar{M}:=M+K.$ 
\end{enumerate}

The following notations will be used throughout the paper:
\begin{itemize}
\item We denote by $(\gk)_{k}$ a sequence satisfying $\gk>\frac{2\bar{M}}{\eta}$ for all $k\in\N$, and $\gk \f\infty\;\hbox{as}\;k \f\infty$. We define the two real sequences $(\a_k)_k$ and $(\sigma_k)_k$ by \begin{equation}\label{sigmadef} \a_{k}:=\frac{\ln \left(\frac{\eta\gk}{2\bar{M}}\right)}{\gk}\;\,\hbox{and}\;\,\sigma_k:=\frac{r\bar{M}_{\psi}}{2\eta^2}\bigg(\frac{\ln(r)}{\gk}+\a_k\bigg),\;\;\;k\in\N.\end{equation}
For $(\a_{k})_k$, we have  $\gk e^{-\a_{k}\gk}=\frac{2\bar{M}}{\eta},$ $\a_k>0$ for all $k\in\N$, $\a_k\searrow$ and $\a_{k}\f 0$. For $(\sigma_k)_k$, we have  $\sigma_k>0$ for all $k\in\N$, $\sigma_k\searrow$ and $\s_{k}\f 0.$
\item  The function $\psi\colon\R^n\f\R$ is defined by
 \begin{equation}\label{psidef} \psi(x):=\max\{\psi_i(x) : i=1,\dots,r\},\;\;\forall x\in\R^n. \end{equation}
Clearly we have that $C=\{x\in\R^n : \psi(x)\leq 0\}.$
\item  For $k\in\N$, we define the function $\psi_{\gk}\colon\R^n\f\R$ by \begin{equation}\label{psigkdef} \psi_{\gk}(x):=\frac{1}{\gk}\ln\Bigg(\sum_{i=1}^{r}e^{\gk \psi_i(x)}\Bigg),\;\;\forall x\in\R^n. \end{equation}
We also define, for each $k\in\N$, \begin{equation}\label{Cgkdef} C^{\gk}:=\{x\in\R^n : \psi_{\gk}(x)\leq 0\}=\bigg\{x\in\R^n : \sum_{i=1}^{r}e^{\gk \psi_i(x)}\leq 1\bigg\},\;\hbox{and}\end{equation}
 \begin{equation}\hspace{-0.4cm}\label{Cgkkdef} C^{\gk}(k):=\{x\in\R^n : \psi_{\gk}(x)\leq -\a_k\}=\bigg\{x\in\R^n : \sum_{i=1}^{r}e^{\gk \psi_i(x)}\leq \frac{2\bar{M}}{\eta\gk}\bigg\}.\end{equation}
 One can easily see that if $r=1$, then $\psi_{\gk}$ and $C^{\gk}$ coincide with $\psi$ and $C$, respectively.
\item For given $x(\cdot)\in W^{1,1}( [0,T]; C)$, $a\ge0$, and $y\in C$, we define the sets
\begin{equation*}\label{Ibtau}
I^{\-}(x):=\bigcap\limits_{i=1}^r  I_i^{\-}(x)= \{t\in [0,T]:  x(t)\in \inte C\},\;\,\hbox{and hence}\;\,I^{0}(x)=[0,T]\setminus I^{\-}(x),
\end{equation*} 
where  $I_i^{\-}(x)$ is the set defined in \eqref{Ibtaui}.
\begin{equation}\label{Iax}\hspace{-0.15cm} \I^a_{y}:=\{i\in\{1,\dots, r\}: -a\le\psi_i(y)\le0\}\;\,\hbox{and}\;\,I^a(x):=\{t\in [0,T]: \I^a_{x(t)}\neq\emptyset\}.
\end{equation}
Note that when  (A2.1) holds, $I^a(x)$ is compact by Lemma \ref{lemmaaux1}.
\item The approximation dynamic $(D_{\gk})$ is defined by  
\begin{equation}\hspace{-0.57cm}\label{Dgk1} ({D}_{\gk}) \begin{sqcases} \dot{x}(t)=f_\Vphi(t,x(t),u(t))-\sum\limits_{i=1}^{r} \gk e^{\gk\psi_i(x(t))} \nabla\psi_i(x(t))\;\;\textnormal{a.e.}\; t\in[0,T],\\ x(0)\in C^{\gk}. \end{sqcases}
 \end{equation}
One can easily verify that the system  $(D_{\gk})$  can be rewritten using the function $\psi_{\gk}$ as the following
 \begin{equation} \label{Dgk2} ({D}_{\gk}) \begin{sqcases} \dot{x}(t)=f_\Vphi(t,x(t),u(t))-\gk e^{\gk\psi_{\gk}(x(t))} \nabla\psi_{\gk}(x(t))\;\;\textnormal{a.e.}\; t\in[0,T],\\ x(0)\in C^{\gk}, \end{sqcases} \end{equation}
where, by \eqref{psigkdef}, 
\begin{equation}\label{gradpsigk}
\nabla\psi_{\gk}(x)= \frac{\sum\limits_{i=1}^r e^{\gk\psi_i(x)} \nabla\psi_i(x)}{\sum\limits_{i=1}^re^{\gk\psi_i(x)}}.
\end{equation}
\end{itemize}
\begin{remark}\label{psiinfo} $(i)$ Under assumption (A2.1), \cite[Proposition 2.3.12]{clarkeold} implies that the function $\psi$ defined in \eqref{psidef}  is locally Lipschitz on $C+\rho B$, and for all $x\in (C+\rho B)$ we have \begin{equation} \label{subpsi} \partial \psi(x)=\Bigg\{\sum_{i\in\mathcal{J}_x}\lambda_i\nabla\psi_i(x) : \l_i\geq 0\;\,\hbox{and}\;\sum_{i\in\mathcal{J}_x}\l_i=1 \Bigg\}, \end{equation}
where $\mathcal{J}_x:=\{i\in\{1,\dots,r\} : \psi_i(x)=\psi(x)\}$, which is  equal to $\I^0_x$ when  $\psi(x)=0$. Hence,  for  $\bar{M}_\psi$ the constant defined before Proposition \ref{prop1},  we have
\begin{equation*} \left\| \zeta\right\|\le \bar{M}_{\psi}, \;\;  \forall c\in C\;\,\hbox{and}\;\,\forall\zeta\in\partial \psi(c). \end{equation*}
$(ii)$ Assumption (A2.2) can be rephrased  by means of $(i)$  in terms of $\psi$ as follows: \\ There is a constant $\eta>0$ such that $$ \left\| \zeta\right\|>2\eta, \;\;  \forall c\in \{x\in \R^n: \psi(x)=0\}\;\,\hbox{and}\;\,\forall\zeta\in\partial \psi(c).$$ 
\end{remark}

\subsection{Properties of $(\psi_{\gk})_k$ and $(C^{\gk}(k))_k$ }\label{propertiesofpsigk} The goal of this subsection is to provide important properties of the functions $\psi_{\gk}$ and the sets $C^{\gk}(k)$ that play an essential role in the construction of the approximating problems $(\tilde{P}_{\gk}^{\ab})$. The proofs of these properties are postponed to Section \ref{auxresults}.

For the sequence of functions $(\psi_{\gk})_k$, we have the following.

\begin{proposition}[Properties of $(\psi_{\gk})_k$] \label{proppsigk} 
 Under \textnormal{(A2.1)-(A2.2)}, the following assertions hold\sp$:$
\begin{enumerate}[$(i)$]
\item The sequence $(\psi_{\gk})_k\in \CO^{1,1}$ on $C+\rho B$, is monotonically nonincreasing in terms of $k$, and converges uniformly to $\psi$. Moreover, for all $k\in\N$, we have 
\begin{equation} \label{pisgkineq} \psi(x)\leq \psi_{\gk}(x)\leq \psi(x)+\frac{\ln(r)}{\gk},\;\;\forall x\in\R^n,\;and
\end{equation}
\begin{equation}\label{gradbdd}
\|\nabla\psi_{\gk}(x)\|\le \bar{M}_\psi, \;\; \forall x\in C.
\end{equation}
\item  There exist $k_1\in\N$ and $r_1>0$ satisfying $2 r_1\leq \rho$, such that for all $k\geq k_1$, for all $x\in \{x\in \R^n: \psi_{\gk}(x)=0\}$, and for all $z\in B_{2r_1}(x)$, we have 
 \begin{equation*}\label{psiuniform}
  \|\nabla\psi_{\gk}(z)\|> 2\eta.
\end{equation*} 
In particular,  for $k\geq k_1$ we have \begin{equation}\label{A2.2forpsigk} [\psi_{\gk}(x)=0]\im \|\nabla\psi_{\gk}(x)\|>2\eta.\end{equation}
\item There exists $k_2\geq k_1$ and $\e_{o}>0$ such that for all $k\geq k_2 $ we have  \begin{equation*}\label{epsilonpsigk} \big[x\in C^{\gk}\;\hbox{and}\; \|\nabla\psi_{\gk}(x)\|\leq \eta\big]\implies \psi_{\gk}(x)<-\e_o.\end{equation*}
\end{enumerate}
\end{proposition}

\begin{remark} \label{cgkproper} From \eqref{pisgkineq} and the definition of $C^{\gk}$ in \eqref{Cgkdef}, we conclude that $C^{\gk}\subset C$ for all $k$. On the other hand, it is easy to prove that under the assumptions \textnormal{(A2.1)-(A2.2)}, and using \eqref{Cgkdef}, there exists $k_3\geq k_2$ such that for all $k\geq k_3$, $C^{\gk}$ is nonempty and compact, with $C^{\gk}\subset\inte C$ for $r>1$. Hence, since by Proposition \ref{proppsigk},  $\psi_{\gk}$ is $C^{1,1}$ on $C^{\gk}+\rho B$ and satisfies \eqref{A2.2forpsigk}, we deduce that all the properties satisfied by the set $C$ in \cite[Lemma 3.3 \& 3.4]{VCpaper} are also satisfied by $C^{\gk}$ for all $k\geq k_3$.
\end{remark}

We proceed to present the properties of the sequence of sets $(C^{\gk}(k))_k$. We denote by $2M_{\psi_{\gk}}$ the Lipschitz constant of $\nabla \psi_{\gk}$ over the compact set $C+\frac{\rho}{2}\bar{B}$ satisfying $M_{\psi_{\gk}}\geq \frac{4\eta}{\rho}$.

\begin{proposition}[Properties of $(C^{\gk}(k))_k$] \label{propcgk(k)} 
 Under \textnormal{(A2.1)-(A2.2)}, the following assertions hold\sp$:$
\begin{enumerate}[$(i)$]
\item  For all $k$, the set $C^{\gk}(k)\subset \inte C^{\gk}\subset \inte C$ and is compact. Moreover, there exists $k_4\geq k_3$ such that for $k\geq k_4$, we have$\sp:$
\begin{enumerate}[$(a)$]
\item $\bdry C^{\gk}(k)=\{x\in\R^n : \psi_{\gk}(x)=-\a_k\}=\Big\{x\in\R^n : \sum\limits_{i=1}^r e^{\gk\psi_i(x)}= \frac{2\bar{M}}{\eta\gk}\Big\}\not=\emptyset.$
\item $\inte C^{\gk}(k)=\{x\in\R^n : \psi_{\gk}(x)<-\a_k\}\not=\emptyset.$
\item $C^{\gk}(k)$ is amenable, epi-Lipschitz with $C^{\gk}(k)=\clo(\inte C^{\gk}(k))$, and $\frac{\eta}{2M_{\psi_{\gk}}}$-prox-regular.
\item For all $x\in\bdry C^{\gk}(k)$ we have $$N_{C^{\gk}(k)}(x)\bbp= \bbp N^P_{C^{\gk}(k)}(x)=N^L_{C^{\gk}(k)}(x)\bp=\bp\bigg\{\l\sum\limits_{i=1}^r e^{\gk\psi_i(x)}\nabla\psi_i(x) : \l\geq 0\bigg\}.$$
\end{enumerate}
\item  The sequence $(C^{\gk}(k))_k$ is a nondecreasing sequence whose  Painlev\'e-Kuratowski limit is $C$ and satisfies
\begin{equation}\label{union.Cgk(k)}\inte C=\bigcup_{k\in \N} \inte C^{\gk}(k)=\bigcup_{k\in \N} C^{\gk}(k). \end{equation}
\item  For $c\in \bdry C$, there exist $\bar{k}_c\geq k_4$, ${r}_{\c}>0$, and a vector $d_c\not=0$ such that  \begin{equation*}\label{lastideahope}\left(\left[C\cap \bar{B}_{{r}_{\c}}({c})\right]+\sigma_k\frac{d_c}{\|d_c\|}\right)\subset \inte C^{\gk}(k),\;\;\forall k\geq \bar{k}_c.\end{equation*}
In particular,  for $k\geq \bar{k}_c$ we have \begin{equation}\label{boundarypoint(k)}
 \left({c}+\sigma_k\frac{d_c}{\|d_c\|}\right)\in \inte C^{\gk}(k). \end{equation}
\end{enumerate}
\end{proposition}

\begin{remark} \label{intelemma(k)} From Proposition \ref{propcgk(k)}$(iii)$, we deduce that for any $c\in C$, there exists a sequence $({c}_{\gk})_k$ such that, for $k$ large enough, ${c}_{\gk}\in \inte C^{\gk}(k)\subset \inte C^{\gk}$, and ${c}_{\gk}\f c$. Indeed: \begin{enumerate}[$(i)$]
\item For $c\in \bdry C$,  we choose $c_{\gk} := c+\sigma_k\frac{d_c}{\|d_c\|}$ for all $k$.  For $k\geq \bar{k}_c$, we have from \eqref{boundarypoint(k)} that $c_{\gk}\in \inte C^{\gk}(k)$. Moreover, since $\sigma_k\f 0$ and $\frac{d_c}{\|d_c\|}$ is a unit vector, we have $c_{\gk}\f c$.
\item For $c\in \inte C$, equation \eqref{union.Cgk(k)} yields the existence of $\hat{k}_c \in \N$,  such that $c\in \inte C^{\gk}(k)$ for all $k\ge \hat{k}_c$.
 Hence, there exists  $ \hat{r}_{c}>0$  satisfying  \begin{equation*} \label{approxxbar}c\in \bar{B}_{\hat{r}_{c}}({c})\subset  \inte C^{\gk}(k),\;\;\forall\sp k\ge \hat{k}_c.\end{equation*}
In this case,  we  take   the sequence ${c}_{\gk}\equiv c\in\inte C^{\gk}(k)$ that converges to $c$. 
\end{enumerate}  \end{remark}

\begin{remark} \label{veraimp} When combining Proposition \ref{proppsigk}$(ii)$ and Proposition \ref{propcgk(k)}$(ii)$, it follows that for $k$ large enough, we have 
\begin{equation*}\label{A2.2forpsigkbis} [\psi_{\gk}(x)=-\a_k]\im \|\nabla\psi_{\gk}(x)\|>\eta.\end{equation*}

\end{remark} 

\subsection{Connection between $(D)$ and $(D_{\gk})$ $-$ Existence of solutions} \label{connection}  We begin this subsection with the following lemma that states that  the unique solution  $x_{\gk}$  of the Cauchy problem corresponding to the dynamic $(D_{\gk})$ with initial condition in $C^{\gk}$ remains in the set $C^{\gk}$ for all $t$, and  the sequence $(x_{\gk})_{k}$ is equicontinuous. Note that since the end-time $T$ was taken to be $1$ in \cite{VCpaper},  when necessary, we shall adjust here  the constants' expressions in the results and proofs extracted from \cite{VCpaper}.

We introduce the  following {\it global} version of (A1): 
\begin{enumerate}[leftmargin=1.24cm, label=\textbf{({A}\arabic*)$_{\G}$}:]
\item For fixed $(x,u)\in C\times \mathbb{U} $, $f(\cdot, x,u)$ is Lebesgue-measurable; and there exists $ M>0$ such that for a.e. $t \in [0,T]$ we have: $(x, u)\mapsto f (t, x, u)$ is continuous on $C\times U(t)$; for all $u\in U(t)$, $x\mapsto f (t, x, u)$ is $M$-Lipschitz on $C$; and  $\|f(t,x,u)\| \leq M $  for all $(x,u)\in  C \times U(t).$
\end{enumerate}

\begin{lemma}[Invariance of $C^{\gk}$ and uniform bounds]\label{invariance} Let \textnormal{(A1)}$_{\G}$, \textnormal{(A2.1)-(A2.2)} and \textnormal{(A2.5)} be satisfied. Then for $k\geq k_3$, the system $(D_{\gamma_k})$ with $x(0)=c_{\gk}\in C^{\gk}$ and $u_{\gk}\in \mathscr{U}$, has a unique solution $x_{\gk}$ which belongs to $W^{1,2}([0,T];\R^n)$ and satisfies  $x_{\gk}(t)\in C^{\gk}$ for all $t\in [0,T]$. Furthermore, the sequence $(x_{\gamma_k})_k$ satisfies
\begin{equation*} \label{eq2} \|x_{\gk}\|_{\infty} \leq \|c_{\gk}\|+\sqrt{\left(\bar{M}T\right)^2+2T}\;\;\;\hbox{and}\;\;\;\int_0^T\|\dot{x}_{\gk}(t)\|^2\sp dt\leq \bar{M}^2T+2,\end{equation*}
and thus,  it is equicontinuous. \end{lemma}
\begin{proof} By Proposition \ref{proppsigk}, we have that for each $k\geq k_3$, $\psi_{\gk}$ satisfies the {\it same assumptions} satisfied by the function $\psi$ of \cite[Lemma 4.1]{VCpaper}. Moreover,  the system $(D_{\gk})$ utilized in \cite[Lemma 4.1]{VCpaper}, in which  $\psi$ is replaced by $\psi_{\gk}$, coincides with our system $(D_{\gk})$, see \eqref{Dgk2}. Therefore, Lemma \ref{invariance} follows immediately from \cite[Lemma 4.1]{VCpaper}, where $\psi$ and $C$ are replaced by $\psi_{\gk}$ and $C^{\gk}$, respectively. We note that in \cite[Lemma 4.1]{VCpaper}, $\|c_{\gk}\|$ is replaced by  a bound $\a_0>0$ of the sequence $(c_{\gk})_k$.   \end{proof}

We denote by $\xi_{\gk}(\cdot)$ the sequence of non-negative continuous functions on $[0,T]$ corresponding to the solution $x_{\gk}$, obtained in Lemma \ref{invariance}, and  defined by
 \begin{equation} \label{defxi} \xi_{\gk}(\cdot):= \gk e^{\gk\psi_{\gk}(x_{\gk}(\cdot))} \overset{\eqref{psigkdef}}{=}\sum_{i=1}^{r}\xi_{\gk}^i(\cdot),\;\hbox{where}\;\xi_{\gk}^i(\cdot):=\gk e^{\gk\psi_i(x_{\gk}(\cdot))},  \; i=1,\dots,r.\end{equation}
The  following result illustrates the tight link between the solutions of $(D_{\gk})$ and those of $(D)$. It is parallel to \cite[Theorem 4.1]{VCpaper}, and it follows using arguments similar to those used in the proofs of \cite[Theorem 4.1]{VCpaper} and \cite[Theorem 4.1]{verachadi}. See also the proof of \cite[Theorem 2.2]{pinhonew} where a simplified method is used for special setting.

\begin{theorem}[Convergence of $(\xi_{\gk})_k$ -- $(D_{\gk})_k$ approximates $(D)$] \label{DgktoD} Assume that \textnormal{(A1)}$_{\G}$, \textnormal{(A2.1)-(A2.2)} and \textnormal{(A2.5)-(A4.1)} hold. Then we have the following\sp$:$\vspace{0.15cm}\\
{\bf (I).} Let  be given  a sequence $(c_{\gk})_k$  in $C^{\gk}$ converging to $x_0 \in C_0$ and  a sequence $(u_{\gk})_k$ in $\mathscr{U}$. For $k\ge k_3$, denote by  $x_{\gk}$  the solution of $(D_{\gk})$  corresponding to $(x(0), u)=(c_{\gk},u_{\gk})$  which is obtained via \textnormal{Lemma \ref{invariance}} and hence,  remains in $C^{\gk}$. Then, the following statements are valid$\sp:$
\begin{enumerate}[$(i)$]
 \item  The sequence $(x_{\gk})_k$  admits a subsequence, we do not relabel, that converges uniformly to some $x \in W^{1,2}([0,T];\R^n)$ whose values are in $C$, and its derivative $\dot{x}_{\gk}$ converges weakly in $L^2$ to $\dot{x}$. Moreover, \begin{equation} \label{eqprop}\|x\|_{\infty}\leq \|x_0\|+\sqrt{\left(\bar{M}T\right)^2+2T}\;\;\hbox{and}\;\;\int_0^T\|\dot{x}(t)\|^2\sp dt\leq \bar{M}^2T+2.\end{equation} 
 \item The sequence $(\|\xi_{\gk}\|^2_2)_k$ is bounded by a positive constant $M_\xi$, and there exists a subsequence of $(\gk)_k$, we do not relabel, such that  the sequence of vector functions  $(\xi^1_{\gk}, \dots, \xi^r_{\gk})_k$  converges weakly in $L^2$ to  a nonnegative vector function \sloppy $(\xi^1,\dots, \xi^r)\in L^2([0,T],\R^r)$, and hence,  $\xi_{\gk}$ converges weakly in $L^2$ to $\xi:=\sum_{i=1}^{r}\xi^i$, with 
\begin{equation}\label{xiprop1} 
 \xi^i(t)=0\;\;\hbox{for all}\;\,t\in I_i^{\-}(x)\;\hbox{and for all}\;\;i\in\{1,\dots r\},\end{equation} 
  \begin{equation}\label{xiprop}\hspace{-.88 in}\hbox{and}\hspace{.15 in}\xi(t)=0\;\;\hbox{for all}\;\,t\in I^{\-}(x).
  \end{equation}
 \item If the subsequence $u_{\gk}(t)\xrightarrow[]{\textnormal{a.e}.\;t\,} u(t)$, where ${u}\in \mathscr{U}$, then $x$ is the unique solution of $(D)$ corresponding to $(x_0,u)$ and the triplet $(x, u, (\xi^i)_{i=1}^r)$ satisfies \crefrange{admissible-P}{eqVera3} stated below.
\item If \textnormal{(A4.3)} holds and $f(t, x, U(t))$ is convex  for $x\in C$ and $t\in [0,T]$ \textnormal{a.e.},  then there exists ${u}\in \mathscr{U}$ such that the limit function $x$ is the unique solution of $(D)$ corresponding to $(x_0,u)$. Moreover, we have 
\begin{equation}\label{admissible-P}\dot{x}(t)= f_\Vphi(t,x(t),u(t))-\sum_{i=1}^r\xi^i(t) \nabla\psi_i(x(t)),\;\;\;t\in[0,T]\; \textnormal{a.e.}\end{equation} 
\begin{equation}\label{xibounded}  \|\xi^i\|_{\infty}\leq\frac{\bar{M}}{2\eta}\;\;\hbox{for}\,\;i=1,\dots,r,\;\,\hbox{and}\;\,\|\xi\|_{\infty}\leq\frac{\bar{M}}{2\eta}. \end{equation}
\begin{equation}\label{eqVera2} \| f_\Vphi(t, x(t), u(t))-\dot{x}(t )\| \le \| f_\Vphi(t, x(t), u(t))\|,\;\;\;t\in[0,T]\; \textnormal{a.e.}\end{equation} 
\begin{equation} \label{eqVera3} \hspace{-1cm}  \< f_\Vphi(t, x(t), u(t))-\dot{x}(t ), z-x(t)\> \leq  \frac{\bar{M}M_\psi}{2 \eta}\left\|z-x(t)\right\|^2, \;\,t\in [0,T]\; \textnormal{a.e.,}\;\,\forall z\in C.  \end{equation}
 \end{enumerate}
{\bf (II).} Conversely,  for given  $x_0\in C_0$ and $\u\in \mathscr{U}$, the system $(D)$ with  $(x(0), u)= (x_0, \u)$ has a unique solution $\x$, which  is the uniform limit of a subsequence of $(x_{\gk})_k$ that is not relabeled, where for each $k$,  $x_{\gk}$ is the solution of the system $(D_{\gk})$ corresponding to $(x(0), u)=(c_{\gk}, \u)$, and $c_{\gk}$ is the sequence of \textnormal{Remark \ref{intelemma(k)}} that corresponds to $c=x_0$. Hence, for $k$ sufficiently large, $x_{\gk}(t)\in C^{\gk}$ for all $t\in[0,T]$. Moreover, for ${\xi}^i$ the $L^2$-weak limit of $\xi^i_{\gk}$ obtained in $(ii)$, for $i=1,\dots,r$, the triplet $(\x,\u, ({\xi}^i)_{i=1}^r)$  satisfies \crefrange{eqprop}{eqVera3}. \end{theorem}

\begin{remark} \label{Dforms} Let  $u\in \mathscr{U}$ and $x\in W^{1,1}([0,T];\R^n)$ with  $x(0)\in C_0$ and $x(t)\in C$ for all $t\in [0,T]$. From the statement and the proof of Theorem \ref{DgktoD}, and from \eqref{subpsi}, we deduce that under (A1)-(A3), the following assertions are equivalent:
\begin{enumerate}[$(i)$]
\item  $x$  is a solution for  $(D)$ corresponding to the control  $u$.
\item  There exists a finite sequence of nonnegative measurable functions $(\xi^i)_{i=1}^r$ such that for $i=1,\dots,r$, $\xi^i(t)=0$ for all $t\in I_i^{\-}(x)$,  $\|\sum_{i=1}^r\xi^i\|_\infty\leq \frac{\bar{M}}{2\eta}$, and $(x, u,({\xi}^i)_{i=1}^r)$ satisfies 
\begin{eqnarray*}\dot{x}(t)&=&f_\Vphi(t,x(t),u(t))-\sum_{i=1}^r\xi^i(t) \nabla\psi_i(x(t)) \\&= & f_\Vphi(t,x(t),u(t))-\sum_{i\in\I^0_{x(t)}}\xi^i(t) \nabla\psi_i(x(t)),\;\;\;t\in[0,T]\; \textnormal{a.e.}\end{eqnarray*}
\item There exists nonnegative measurable function $\xi$   such that  $\xi(t)=0$ for all $t\in I^{\-}(x)$ and 
   $$\dot{x}(t)\in f_\Vphi(t,x(t),u(t))-\xi(t)\partial\psi(x(t)),\;\;\;t\in[0,T]\; \textnormal{a.e.} $$
\end{enumerate}
The implication $(iii)\implies(ii)$, is established by taking  $\xi_i(t):= \xi (t)\l_i(x(t))$ for $ i\in \I^0_{x(t)}$, where $(\lambda_i(x(t)))_{ i\in \I^0_{x(t)}}$   are the coefficients of the convex combination  of $\nabla\psi_i (x(t))$'s obtained via \eqref{subpsi} for the element in  ($iii$) belonging to $\partial \psi(x(t))$. \end{remark}

\begin{remark}\label{newproof} Note that   when establishing Theorem \ref{DgktoD}\sp{(I)}$(iii)\& (iv)$,  the proof that $(x,u, (\xi^i)_{i=1}^r)$ satisfies \eqref{admissible-P}  uses  arguments independent of  $\xi_{\gk}^i$ being obtained  through \eqref{defxi}, and hence, that proof is valid for $(\xi_{\gk}^1,\dots,\xi_{\gk}^r)_k$  being any sequence of  $L^2([0,T];\R^r)$-functions that converges weakly in $L^2$ to $(\xi^1,\dots,\xi^r)$, as $k\f\infty$. Therefore,  whenever $(x_j,u_j,(\xi_j^i)_{i=1}^r)_j$  is a sequence solving \eqref{admissible-P} with $x_j$ converging uniformly to $x$ and $(\xi^1_j,\dots,\xi^r_j)$ converging weakly in $L^2$ to $(\xi^1,\dots,\xi^r)$, we have  that   $(x,u,(\xi^i)_{i=1}^r)$ satisfies \eqref{admissible-P} for some $u\in\mathcal{U}$. This function $u$ is   the  almost everywhere pointwise limit of $(u_j)_j$, whenever this limit exists.  Otherwise,  $u$ is obtained  through  the Filippov selection theorem,  by assuming that (A4.3) holds and $f(t, x, U(t))$ is convex  for $x\in C$ and $t\in [0,T]$ a.e. \end{remark}

The following theorem presents  important extra features resulting from starting the solutions of $(D_{\gk})$ in Theorem \ref{DgktoD}{(I)} from the subset $C^{\gk}(k)\subset  \inte C^{\gk}\subset \inte C$ defined in \eqref{Cgkkdef}.

\begin{theorem}[Extra properties of $(x_{\gk})_k$ and $(\xi_{\gk})_k$ when $x_{\gk}(0)\in C^{\gk}(k)$] \label{DgktoD(k)} Assume that \textnormal{(A1)}$_{\G}$, \textnormal{(A2.1)-(A2.2)} and  \textnormal{(A2.5)-(A4.1)} hold. Let $(c_{\gk})_k$ be a sequence such that $c_{\gk}\in C^{\gk}(k)$, for $k$ sufficiently large. Then there exists $k_5\geq k_4$ such that for all sequences $(u_{\gk})_k$ in $\mathscr{U}$ and for all $k\geq k_5$, the solution $x_{\gk}$ of $(D_{\gk})$ corresponding to $(c_{\gk},u_{\gk})$ satisfies\sp$:$
\begin{enumerate}[$(i)$]
\item $x_{\gk}(t)\in C^{\gk}(k)\subset \inte C^{\gk}\subset \inte C$ for all $t\in[0,T]$.
\item $0\le \xi_{\gk}^i(t) \leq \xi_{\gk}(t)\le \frac{2\bar{M}}{\eta}$ for all $t\in [0,T]$ and for $i=1,\dots,r.$
\item $ \|\dot{x}_{\gk}(t)\|\leq \bar{M}+\frac{2\bar{M}\bar{M}_\psi}{\eta}$ \,for \textnormal{a.e.} $t\in [0,T]$.
\end{enumerate}
\end{theorem}
\begin{proof} By Proposition \ref{proppsigk} and Remark \ref{veraimp}, we have that,  for each $k\geq k_3$, the function $\psi_{\gk}$ satisfies the {\it same assumptions} imposed  on the function $\psi$ in  \cite[Theorem 5.1]{VCpaper}, where the system $(D_{\gk})$, in which  $\psi$ is replaced by $\psi_{\gk}$, coincides with our system $(D_{\gk})$.  Thus, Theorem \ref{DgktoD(k)}$(i)$-$(iii)$ follows immediately from \cite[Theorem 5.1$(i)$-$(iii)$]{VCpaper}, where $\psi$ and $C$ are replaced by $\psi_{\gk}$ and $C^{\gk}$, respectively. \end{proof}
 
As a direct consequence of Theorem \ref{DgktoD(k)} and parallel to \cite[Corollary 5.2]{VCpaper}, we have the following.

\begin{corollary}[$\bar{x}_{\gk}$, corresponding to $\x$, in $ C^{\gk}(k)$] \label{coroDgktoD(k)} Assume that \textnormal{(A1)}$_{\G}$, \textnormal{(A2.1)-(A2.2)}, and \textnormal{(A2.5)-(A4.1)} hold. Let  $(\x, \u)$ be a solution of $(D)$. Consider $(\bar{c}_{\gk})_k$ the sequence obtained from \textnormal{Remark \ref{intelemma(k)}} for $c:=\bar{x}(0)$, and  $\x_{\gk}$ the solution of $(D_{\gk})$ corresponding to $(\bar{c}_{\gk},\u)$. Then, there exists $k_7\geq k_6$ such that for all $k\geq k_7$ we have$\sp:$
\begin{enumerate}[$(i)$]
\item $\x_{\gk}(t)\in C^{\gk}(k)\subset \inte C^{\gk}\subset \inte C$ for all $t\in[0,T]$.
\item $0\le \xi_{\gk}^i(t) \leq \xi_{\gk}(t)\le \frac{2\bar{M}}{\eta}$ for all $t\in [0,T]$ and for $i=1,\dots,r.$
\item $ \|\dot{\x}_{\gk}(t)\|\leq \bar{M}+\frac{2\bar{M}\bar{M}_\psi}{\eta}$ \,for \textnormal{a.e.} $t\in [0,T]$.
\end{enumerate}
\end{corollary}

\subsection{Existence of solutions for $(P)$} The following is an existence of optimal solution result for the problem $(P)$. 

\begin{proposition}[Existence of optimal solution for $(P)$] \label{existencesolP} Assume that \textnormal{(A1)}$_{\G}$, \textnormal{(A2.1)-(A2.2)} and  \textnormal{(A2.5)-(A4)} are satisfied, $g\colon \R^{n}\times\R^n\to \R\cup\{\infty\}$ is lower semicontinuous, and $f(t,x,U(t))$ is convex for all $x\in C$ and $t\in [0,T]$ \textnormal{a.e.}  Then the problem $(P)$ has an optimal solution if and only if it has at least one admissible pair $(z_o,u_o)$  with $(z_o(0),z_o(T))\in \dom g$.
\end{proposition}

\begin{proof}  Since any admissible pair $(x,u)$ satisfies $(x(0), x(T))\in C_0\times (C_T\cap C)$, then, by the lower semicontinuity of $g$, the compactness of $C_0\times (C_T\cap C)$, and the existence of the admissible pair $(z_o,u_o)$ with $(z_o(0),z_o(T))\in (\dom g)\cap [C_0\times (C_T\cap C)]$, we have that the $\inf_{(x,u)} (P)$ is finite, and hence, ($P$) admits  a minimizing sequence $(x_n, u_n)$. It follows that $x_n$ satisfies \eqref{eqprop} with $x_n(t)\in C$ for all $t\in [0,T]$. Thus, Arzela-Ascoli’s theorem produces  a subsequence  of  $(x_n)_n$, we do not relabel, that converges uniformly to an absolutely continuous $\x,$ with $\x(t)\in C$  for all $t\in [0,T]$. The equivalence between $(i)$ and  $(ii)$ in Remark \ref{Dforms} produces a sequence of $r$-tuple nonnegative measurable functions $(\xi^1_n,\dots,\xi^r_n)_n$ such that, for $i=1,\dots,r$, $\xi_n^i(t)=0$ for all $t\in I_i^{\-}(x)$ and $\|\xi_n^i\|_\infty\leq \frac{\bar{M}}{2\eta}$, and $(x_n, u_n,(\xi_n^i)_{i=1}^r)$ satisfies \eqref{admissible-P}. Consequently, the sequence $(\xi^1_n,\dots,\xi^r_n)_n$ admits a subsequence, we do not relabel,  that converges weakly in $L^2$ to $(\xi^1,\dots,\xi^r)$.   Remark \ref{newproof}  now asserts the existence of a $\u\in\mathcal{U}$ such that  the pair $(\x,\u)$ solves $(D)$. Since $C_T$ is closed and $x_n(1)\in C_T$ for all $n\in \N$, it follows that $\x(T)\in C_T$, and hence,  ($\x,\u$) is admissible for ($P$). On the other hand,  the lower semicontinuity of $g$ implies $$g(\x(0),\x(T)) \le \liminf_{n\to \infty}g(x_n(0),x_n(T))= \inf_{(x,u)} (P).$$ This proves that $(\x,\u)$ is an optimal solution for ($P$).
\end{proof}

\subsection{Approximating problems for $(P)$} \label{pmp} Let $(\x,\u)$ is a {\it strong local minimizer} for $(P)$ with a corresponding  $\delta>0$. The ultimate goal here is to formulate a suitable sequence of {\it standard} optimal control problems  that approximates the problem $(P)$ and whose optimal solutions form a sequence that approximates $(\x,\u)$. The following lemma states that when (A1) and (A6) are satisfied, the function $f$ can be extended to a function $\tilde{f}$ defined on $[0,T]\times\R^n\times \R^m$ and satisfying (A1)$_{\G}$ with  $\tilde{f}(t,x,U(t))$ convex for all $x\in C$ and $t\in[0,T]$ a.e. The proof of this lemma is not displayed here, as it mimics arguments  used at the beginning of \cite[Proof of Theorem 6.2]{VCpaper}, where  such an extension is obtained for the case in which  $U(\cdot)$ is a  {\it constant} multifunction.

\begin{lemma}[Extension of $f$] \label{fextension} Assume that \textnormal{(A1)-(A2.2)}  hold. Then there exists a function $\tilde{f}\colon[0,T]\times\R^n \times \R^m\f\R^n$ such that\sp$:$ \begin{itemize}
\item  For $t\in [0,T]$ \textnormal{a.e.,} we have $$\tilde{f}(t,x,u)=f(t,x,u),\;\;\forall (x, u)\in (C\cap \bar{B}_{\delta}(\x(t)))\times U(t).$$
\item The function $\tilde{f}$ satisfies  \textnormal{(A1)}$_{\G}$ for $M:=2M_\ell$.
\item If also \textnormal{(A6)} is satisfied, then $\tilde{f}(t,x,U(t))$ is convex for all $x\in C$ and $t\in[0,T]$ \textnormal{a.e.}
\end{itemize}
\end{lemma}

Now let $\delta_{\bbbp o}>0$ be fixed such that \begin{equation*}\label{delta0def} \delta_{\bbbp o}\leq \begin{cases}\min\big\{\hat{r}_{\x(0)},\frac{\delta}{2}\big\}&\qquad\;\hbox{if}\;\x(0)\in\inte C,\vspace{0.1cm}\\ \min\big\{r_{\x(0)},\frac{\delta}{2}\big\}&\qquad\;\hbox{if}\;\x(0)\in\bdry C,\end{cases} \end{equation*}
where $\hat{r}_{\x(0)}$ and  ${r}_{\x(0)}$ are, respectively,  the constants in  Remark \ref{intelemma(k)}$(ii)$  and Proposition \ref{propcgk(k)}$(iii)$  corresponding to $c:=\x(0)$. We introduce the following notations:
\begin{itemize}
\item The trajectory $\x_{\gk}$ is the solution of ($D_{\gk}$) obtained via Corollary \ref{coroDgktoD(k)}. That is, it corresponds to ($\bar{c}_{\gk}, \u$), where $\bar{c}_{\gk}$ is defined via Remark \ref{intelemma(k)} for  $c:=\x(0)$, and hence, $\bar{c}_{\gk}\in C^{\gk}(k)$ for $k$ sufficiently large.
 \item The sequence of sets $(C_0^{\gk}(k))_k$,  is defined by  \begin{equation*}\label{c0kdef} C_0^{\gk}(k):= 
 \begin{cases}C_0\cap \bar{B}_{\delta_{\bbbp o}}\bp\left({\x(0)}\right) ,\;\forall k\in\N, &\;\hbox{if}\;\x(0)\in\inte C,\vspace{0.1cm}\\ \left[C_0\cap\bar{B}_{\delta_{\bbbp o}}\bp\left({\x(0)}\right)\right]+\sigma_k\frac{d_{\x(0)}}{\|d_{\x(0)}\|},\;\forall k\in\N,&\;\hbox{if}\; \x(0)\in\bdry C.\end{cases}\end{equation*} 
\item The sequence of sets $(C_T(k))_k$ is defined by \begin{equation*} \label{C1kdefinition}  C_T(k):=\left[\left(C_T\cap\bar{B}_{\delta_{\bbbp o}}\bp({\x(T)})\right) -\x(T)+{\x}_{\gk}(T)\right]\cap C,\;\;\;k\in\N.\end{equation*} 
\item The function $L\colon[0,T]\times \R^n\f\R$  is defined by 
\begin{equation} \label{Ltx} L(t,x):=\max\left\{\|x-\x(t)\|^2-\frac{\delta^2}{4},0\right\},\;\;\forall (t,x)\in [0,T]\times C. \end{equation} 
It is easy to see that ${L}(\cdot,x)$ is measurable, and ${L}(t,\cdot)$ is Lipschitz on $C$ uniformly in $t\in[0,T]$.
 \end{itemize}

A well known optimal control technique (see e.g., \cite{vinter}) is used to show that a strong local minimizer of an optimal control problem is a global minimizer for another problem,  obtained by adding  to the objective function the integral of the function $L$, defined in \eqref{Ltx},   and  by localizing the endpoints  constraints.  This technique is also employed  in \cite[Proof of Theorem 6.2]{VCpaper} to prove  the following lemma, which states  that $(\x,\u)$  is a {\it global} minimizer for the new problem $(\tilde{P})$. 

\begin{lemma}[$(\x,\u)$ global solution for $(\tilde{P})$] \label{localtoglobal} Assume that \textnormal{(A1)-(A2.2)} hold, and let $\tilde{f}$ be the extension of $f$ obtained in \textnormal{Lemma \ref{fextension}}. Assume further that \textnormal{(A4.1)-(A4.2)} are satisfied and $g$ is continuous on $\tilde{C}_0(\delta)\times\tilde{C}_T(\delta)$ for some $\tilde{\rho}>0$. Then, $(\x,\u)$ is a global minimizer for the problem $$\begin{array}{l} (\tilde{P})\colon\;  \textnormal{Minimize}\;g(x(0),x(T))+\tilde{K}\int_0^T L(t,x(t))\sp dt\vspace{0.1cm}\\ \hspace{0.9cm}  \textnormal{over}\;(x,u)\;\textnormal{such that}\;u(\cdot)\in \mathscr{U}\;\textnormal{and}\\[2pt] \hspace{0.9cm} \begin{cases} (\tilde{D}) \begin{sqcases}\dot{x}(t)\in \tilde{f}(t,x(t),u(t))-\partial \varphi(x(t)),\;\;\hbox{a.e.}\;t\in[0,T],\\x(0)\in C_0\cap \bar{B}_{\delta_{\bbbp o}}(\x(0)), \end{sqcases}\vspace{0.1cm}\\x(T)\in C_T\cap \bar{B}_{\delta_{\bbbp o}}(\x(T)), \end{cases}
    \end{array}$$ 
    where $\tilde{K}:=\tfrac{512\;\bar{M}_{\ell} M_g}{5\;\delta^3}$, $\bar{M}_{\ell}:= 2{M}_{\ell}+K$ and $M_g:=\max\limits_{\scaleto{\tilde{C}_0(\delta)\times \tilde{C}_T(\delta)}{6.5pt}} g(x,y)$. \end{lemma}

Now we are ready to state our approximation result which follows using arguments similar to those used in the proof of \cite[Proposition 6.2]{VCpaper}.

\begin{proposition} [Approximating problems for $(P)$]\label{prop3old} Assume that \textnormal{(A1)-(A2.2)},  \textnormal{(A2.5)-(A4)} hold, $g$ is continuous on $\tilde{C}_0(\delta)\times\tilde{C}_T(\delta)$ for some $\tilde{\rho}>0$, and $f(t,x,U(t))$ is convex for all $x\in C$ and  $t\in [0,T]$ \textnormal{a.e.}  Let  $\tilde{f}$ be the extension of $f$ in \textnormal{Lemma \ref{fextension}},  and $\tilde{K}$ be the constant in  \textnormal{Lemma \ref{localtoglobal}}. Then,  for every $\a> 0$ and $\beta\in (0,1]$, there exist a subsequence of $(\gk)_k$, we do not relabel, and a sequence $(c_{\gk}, e_{\gk}, u_{\gk})_k$ in $ C_0^{\gk}(k)\times C_T(k) \times\mathscr{U}$ such that,
for\sp\sp$:$
 \begin{equation*}\label{fbeta} \tilde{f}^\beta (t,x,u):=(1-\b) \tilde{f}(t,x,\u(t))+\b \tilde{f}(t,x,u),\;\forall t\in[0,T]\;\textnormal{a.e.},\;\forall (x,u)\in \R^n\times U,\end{equation*} 
  \begin{equation*}\label{fbetaphi} \tilde{f}_\Vphi^\b(t,x,u):= \tilde{f}^\b (t,x,u)-\nabla \Vphi(x),\;\forall t\in[0,T]\;\textnormal{a.e.},\;\forall (x,u)\in \R^n\times U,\end{equation*}
\begin{equation*}
J(x,u):= g(x(0),x(T))+\tilde{K}\int_0^T L(t,x(t))\sp dt+\a\left(\|u(t)-u_{\gk}(t)\|_1+\|x(0)-c_{\gk}\|+\|x(T)-e_{\gk}\|\right),\end{equation*}
and \begin{equation*}\label{dgkba} \begin{array}{l} (\tilde{P}_{\gk}^{\ab})\colon\; \textnormal{Minimize}\;J(x,u)\\ \hspace{1.15cm} \textnormal{over}\;(x,u)\;\textnormal{such that}\;u(\cdot)\in \mathscr{U}\;\textnormal{and}\\[1pt] \hspace{1.11cm} \begin{cases} (\tilde{D}^\b_{\gk})\begin{sqcases} \dot{x}(t)=\tilde{f}^\b_\Vphi(t,x(t),u(t))-\sum\limits_{i=1}^{r}\gk e^{\gk\psi_i(x(t))} \nabla\psi_i(x(t)),\;\;\textnormal{a.e.}\; t\in[0,T],\\x(0)\in C_0^{\gk}(k),\end{sqcases} \vspace{0.1cm}\\ x(T)\in C_T(k),\end{cases}
 \end{array}\end{equation*}
 the pair  $(x_{\gk},u_{\gk})$ that solves $(\tilde{D}^\b_{\gk})$ for $x_{\gk}(0)=c_{\gk}$, is optimal for  $(\tilde{P}_{\gk}^{\ab})$ and satisfies $x_{\gk}(T)=e_{\gk}$. Moreover, we have $$u_{\gk}\xrightarrow[L^1]{\textnormal{strongly}}\u, \;\;\;x_{\gk}\xrightarrow[]{\textnormal{uniformly}}\x,$$
and all the conclusions of \textnormal{Theorem \ref{DgktoD(k)}}  are valid, including $(x_{\gk})_k$ is uniformly Lipschitz and $x_{\gk}(t)\in C^{\gk}(k)\subset \inte C^{\gk}\subset \inte C$ for all $t\in[0,T]$. Furthermore, for all $k$ sufficiently large, we have $x_{\gk}(0)\in (C_0+\tilde{\rho}{B})\cap \inte C$ and $x_{\gk}(T)\in (C_T+\tilde{\rho}{B})\cap \inte C$.
\end{proposition} 

The next proposition is obtained as a direct application of the {\it nonsmooth} Pontryagin maximum principle (e.g., \cite[Theorem 6.2.1]{vinter}) to each member of the family of approximating problems $(\tilde{P}_{\gk}^{\ab})$, obtained in Proposition \ref{prop3old}. Note that the function $\tilde{f}$ is replaced  by $f$ in the statement of the proposition,  since we have the following:
\begin{itemize}
\item $\tilde{f}(t,x,u)=f(t,x,u)$, for $t\in [0,T]$ a.e. and for all $(x, u)\in (C\cap \bar{B}_{\delta}(\x(t)))\times U(t)$.
\item The sequence  $(x_{\gk})_k$ of Proposition \ref{prop3old}  belongs to $\inte C$ and converges uniformly to $\x$. Hence, for $k$ large enough, $x_{\gk}(t)\in(\inte C)\cap B_{\delta_{\bbbp o}}(\x(t))\subset \inte (C\cap \bar{B}_{\delta}(\x(t))) $ for all $t\in[0,T]$.
\end{itemize}

\begin{proposition}[Maximum Principle for the approximating problems] \label{prop4old}  Let $\alpha >0$ and $\b\in(0,1]$ be fixed. Assume that {\textnormal{(A1)-(A2.2)}} and \textnormal{(A2.5)-(A6)} hold. Let  $(x_{\gk},u_{\gk})$ be the sequence in {\textnormal{Proposition \ref{prop3old}}} that is optimal for $(\tilde{P}_{\gk}^{\ab})$ with $x_{\gk}$ converging uniformly to $x_{\gk}$ and $u_{\gk}$ converging strongly in $L^1$ to $\u$. Then, for $k\in\N$, there exist $p_{\gk}\in W^{1,1}([0,T];\R^n)$ and $\l_{\gk}\geq 0$ such that\sp$:$ \vspace{0.1cm}
\begin{enumerate}[$(i)$]
\item {\bf (The nontriviality condition)} For all $k\in\N$, we have \begin{equation*}\label{ntlk} \|p_{\gk}(T)\| + \l_{\gk}=1;
\end{equation*}
\item {\bf (The adjoint equation)} For $t\in [0,T]$ \textnormal{a.e.,} \begin{eqnarray} \label{adjapp} \dot{p}_{\gk}(t)\in &-&(1-\beta)(\partial^{\sp x} {f}(t,x_{\gk}(t),\u(t)))\tran p_{\gk}(t)-\b(\partial^{\sp x} {f}(t,x_{\gk}(t),u_{\gk}(t)))\tran p_{\gk}(t)\\ \nonumber&+&\partial^2\Vphi(x_{\gk}(t))p_{\gk}(t)+\sum_{i=1}^{r} \gk e^{\gk \psi_i(x_{\gk}(t))}\partial^2\psi_i(x_{\gk}(t))p_{\gk}(t)\\ &+&  \nonumber \label{mp4} \sum_{i=1}^{r} \gk^2 e^{\gk \psi_i(x_{\gk}(t))}\nabla\psi_i(x_{\gk}(t))\<\nabla\psi_i(x_{\gk}(t)),p_{\gk}(t)\> + \l_{\gk}\partial^{\sp x}L(t, x_{\gk}(t)); \nonumber\end{eqnarray}
\item {\bf (The transversality equation)} 
\begin{equation*}\label{mp5bis} \bp\bp (p_{\gk}(0),-p_{\gk}(T))\bp \in\bp \l_{\gk}\partial^L g(x_{\gk}(0),x_{\gk}(T)) +\big[\big(\a\bar{B}+N_{C_0^{\gk}(k)}^L(x_{\gk}(0))\big)\times \big(\a\bar{B}+  N_{C_T(k)}^L(x_{\gk}(T))\big)\big]; \end{equation*}
\item {\bf (The maximization condition)} \begin{equation*}\label{maxapp} \max_{u\in U}\left\{\<{f}^{\b}_\Vphi(t,x_{\gk}(t),u),p_{\gk}(t)\>-\frac{\a\l_{\gk}}{\b}\|u-u_{\gk}(t)\|\right\}\;\hbox{is attained at}\;u_{\gk}(t)\;\,\hbox{for}\;\,t\in [0,T]\;\,\textnormal{a.e.}\end{equation*}
\end{enumerate}
Furthermore, if $C_T=\R^n$, then $\l_{\gk}\not=0$ and is taken to be $1$, and the nontriviality condition $(i)$ is eliminated.
\end{proposition}

\section{Proof of Theorem \ref{thm1mpcomp}} \label{proofthm1mp} We first prove the theorem under the temporary additional assumption \enquote{$C$ is compact} which is {\it stronger} than (A2.4), see Remark \ref{assumptionH}. The removal of this additional assumption will constitute the final step in the proof. 

Let $(\x,\u)$ be a strong local minimizer for $(P)$. For each fixed $(\a, \b)\in (0,1]^2:=(0,1]\times (0,1]$, Propositions  \ref{prop3old} and \ref{prop4old} produce a subsequence of $(\gk)_k$, we do not relabel, and corresponding sequences  $x_{\gk}$, $u_{\gk}$, $p_{\gk}$ and $\l_{\gk}$ such that: 
\begin{itemize}
\item For each $k$, the admissible pair $(x_{\gk},u_{\gk})$ is optimal for $(\tilde{P}_{\gk}^{\ab})$.
\item $u_{\gk}$ converges strongly in $L^1$ to $\u$, and  $u_{\gk}(t) \f \u(t)$ a.e.
\item $x_{\gk}$ converges uniformly to $\x$.
\item  All the conclusions of \textnormal{Theorem \ref{DgktoD(k)}}  are valid, including $(x_{\gk})_k$ is uniformly Lipschitz and $x_{\gk}(t)\in\inte C$ for all $t\in[0,T]$.
\item For all $k$, $x_{\gk}(0)\in (C_0+\tilde{\rho}{B})\cap \inte C$ and $x_{\gk}(T)\in (C_T+\tilde{\rho}{B})\cap \inte C$.
\item Equations  \ref{ntlk}--\ref{maxapp} of Proposition \ref{prop4old} are satisfied.
\end{itemize}
Since $x_{\gk}(t)\in\inte C$ for all $t\in[0,T]$, we have,  for $t\in [0,T]$ a.e.,  
$$\partial^{\sp x} f(t,x_{\gk}(t),u_{\gk}(t))=\partial^{\sp x}_\ell f(t,x_{\gk}(t),u_{\gk}(t)),\;\;\partial^2\Vphi(x_{\gk}(t))=\partial^2_\ell\varphi(x_{\gk}(t)),$$
$$\partial^{\sp x} f(t,x_{\gk}(t),\u(t))=\partial^{\sp x}_\ell f(t,x_{\gk}(t),\u(t)),$$
$$\partial^2\psi_i(x_{\gk}(t))= \partial^2_\ell\psi_i(x_{\gk}(t))\;\;\hbox{for}\;\;i=1,\dots,r, $$
\begin{equation*}\label{dllg}\hbox{and}\;\;\partial^{\sp L} g(x_{\gk}(0),x_{\gk}(T))=\partial^{\sp L}_{\ell} g(x_{\gk}(0),x_{\gk}(T)).\vspace{0.1 cm}\end{equation*}
Hence from (\ref{adjapp}), it results the existence of sequences $\zeta_{\gk},\, \hzeta_{\gk},\,\theta_{\gk}$ and $(\vartheta^i_{\gk})_{i=1}^r$ in $\mathscr{M}_n([0,T])$, and $\omega_{\gk}\colon[0,T]\f\R^n$
such that, for a.e. $t\in [0,T],$
\begin{equation*} (\hzeta_{\gk}(t),\zeta_{\gk}(t),\theta_{\gk}(t))\in \partial^{\sp x}_\ell f(t,x_{\gk}(t),u_{\gk}(t))\times \partial^{\sp x}_\ell f(t,x_{\gk}(t),\u(t))\times  \partial^2_\ell\varphi(x_{\gk}(t)), \end{equation*}
\begin{equation*} \vartheta^i_{\gk}(t))\in \partial^2_\ell\psi_i(x_{\gk}(t))\;\;\hbox{for}\;\;i=1,\dots,r,\end{equation*}
\begin{equation}\label{omegabounded}\omega_{\gk}(t)\in \partial^{\sp x}L(t,x_{\gk}(t))= \conv\{2(x_{\gk}(t)-\x(t)), 0\} =[0, 2(x_{\gk}(t)-\x(t))], \end{equation}
\begin{eqnarray}\label{mpim1} \hbox{and}&&\bp\bp\bp\dot{p}_{\gk}(t)= - (1-\beta)\zeta_{\gk}(t)\tran p_{\gk}(t)- \beta\sp \hzeta_{\gk}(t)\tran p_{\gk}(t)+ \theta_{\gk}(t)p_{\gk}(t)\\[3pt]&+&\nonumber  \sum_{i=1}^r\gk e^{\gk \psi_i(x_{\gk}(t))}\vartheta^i_{\gk}(t)p_{\gk}(t)\\ \nonumber &+& \sum_{i=1}^r \gk^2 e^{\gk \psi_i(x_{\gk}(t))}\nabla\psi_i(x_{\gk}(t))\<\nabla\psi_i(x_{\gk}(t)),p_{\gk}(t)\>+\l_{\gk}\omega_{\gk}(t)\nonumber.\end{eqnarray} 
Note that for each $k$, the functions $p_{\gk}, \dot{p}_{\gk}, x_{\gk}$, $u_{\gk}$ and $\u$ are measurable on $[0,T]$, and, by  \cite[Exercise 13.24]{clarkebook} and (A1), (A2.1), and (A3), the multifunctions  
$\partial_{\ell}^{\sp x} f(\cdot,x_{\gk}(\cdot),u_{\gk}(\cdot))$, $\partial_{\ell}^{\sp x} f(\cdot,x_{\gk}(\cdot),\u(\cdot))$, $\partial_{\ell}^2\varphi(x_{\gk}(\cdot))$, $(\partial_{\ell}^2\psi_i(x_{\gk}(\cdot)))_{i=1}^r$, and $\partial^xL(\cdot,x_{\gk}(\cdot))$ are measurable on $[0,T]$.  Hence,  the Filippov measurable selection theorem (see \cite[Theorem 2.3.13]{vinter}) yields that we can assume the measurability of the functions $\zeta_{\gk}(\cdot)$, $\hzeta_{\gk}(\cdot)$, $\theta_{\gk}(\cdot)$,  $(\vartheta_{\gk}^i(\cdot))_{i=1}^r$, and $\omega_{\gk}(\cdot)$. Moreover,  these sequences are uniformly bounded in $L^{\infty}$, as  $\|\zeta_{\gk}\|_\infty \leq {M_\ell}$, $\|\hzeta_{\gk}\|_\infty \leq {M_\ell}$, $\|\theta_{\gk}\|_\infty \leq K$, $\|\vartheta^i_{\gk}\|_\infty\leq 2M_{\psi}$ for $i=1,\dots,r$, and \begin{equation}\label{ML} \|\omega_{\gk}\|_\infty\leq \overbrace{2\|x_{\gk}-\x\|_\infty}^{d_{\gk}}\leq M_L,\end{equation} for some $M_L>0$, which exists from \eqref{omegabounded} and the uniform convergence of the sequence $x_{\gk}$ to $\x$.

The  major difference between the proof of this theorem and that of \cite[Theorem 6.2]{VCpaper} lies in the construction of the adjoint vector $p$ (for fixed $(\a,\beta)\in (0,1]^2$). More specifically, the difference is manifested below in the intricacy to prove that in our setting, where  $r>1$,   the sequence $(p_{\gk})_{k}$  enjoys the same properties  obtained there for  $r=1$,  namely,   $(p_{\gk})_{k}$ is uniformly bounded and has uniform bounded variation. Once we establish these facts for our general  setting,    the construction of the remaining items $\nu^i$ (for each $i$), $\zeta$, $\theta$, $\vartheta_i$ (for each $i$) and $\l$, follows arguments similar to those in \cite[Proof of Theorem 6.2]{VCpaper} when constructing therein  $\nu$, $\xi$, $\zeta$, $\theta$, $\vartheta$ and $\l$, respectively.

We fix $(\a,\beta)\in (0,1]^2$. Using (\ref{mpim1}), we obtain 
\begin{eqnarray*}&&\bp\bp\bp \frac{1}{2}\frac{d}{dt}\|p_{\gk}(t)\|^2=\\ &&\left\<\left(-(1-\b)\zeta_{\gk}(t)\tran- \beta\sp \hzeta_{\gk}(t)\tran+\theta_{\gk}(t)+\sum_{i=1}^r\gk e^{\gk \psi_i(x_{\gk}(t))}\vartheta^i_{\gk}(t)\right)p_{\gk}(t),p_{\gk}(t)\right\> \nonumber \\ &+&\sum_{i=1}^r \gk^2 e^{\gk \psi_i(x_{\gk}(t))}\left|\left\<\nabla\psi_i(x_{\gk}(t)),p_{\gk}(t)\right\>\right|^2+\l_{\gk}\<\omega_{\gk},p_{\gk}\> \\&\geq& \left\<\left(-(1-\b)\zeta_{\gk}(t)\tran- \beta\sp \hzeta_{\gk}(t)\tran+\theta_{\gk}(t)+\sum_{i=1}^r\gk e^{\gk \psi_i(x_{\gk}(t))}\vartheta^i_{\gk}(t)\right)p_{\gk}(t),p_{\gk}(t)\right\>\\&+& \l_{\gk}\<\omega_{\gk},p_{\gk}\>\\ &\geq&  -\left({M}_\ell+K+2M_{\psi}\sum_{i=1}^r\gk e^{\gk \psi_i(x_{\gk}(t))}\right)\|p_{\gk}(t)\|^2 - 2\l_{\gk}\|x_{\gk}-\x\|_{\infty} \|p_{\gk}(t)\| \\&=& -\left({M}_\ell+K+2M_{\psi}\xi_{\gk}(t)\right)\|p_{\gk}(t)\|^2 -\l_{\gk}d_{\gk}\|p_{\gk}(t)\|.\end{eqnarray*}	
Hence, using Gr\"onwall's Lemma in \cite[Lemma 4.1]{show} for $\a=\frac{1}{2}$, and the uniform boundedness of $(\|\xi_{\gk}\|^2_2)_k$ proved in Theorem \ref{DgktoD}, we obtain that 
\begin{equation}\label{mp7} \|p_{\gk}(t)\|\leq e^{(M_\ell+K+2M_{\psi}\sqrt{M_\xi})}(\|p_{\gk}(T)\|+\l_{\gk}d_{\gk})\leq (1+L_g+M_L) e^{(M_\ell+K+2M_{\psi}\sqrt{M_\xi})}:=M_1, \end{equation}
where the last inequality follows from  \eqref{ML}, and the nontriviality condition \ref{ntlk} when $C_T\neq\R^n$, and the transversality condition \ref{mp5bis} when $C_T=\R^n$. Hence, $(p_{\gk})$ is uniformly bounded by the constant $M_1$ on $[0,T].$

We proceed to prove that $(\dot{p}_{\gk})$ is uniformly bounded in $L^1$. First, we note that the method used in \cite{VCpaper}, for $r=1$, cannot be used for the version \eqref{Dgk2} of $(D_{\gk})$ since the generalized Hessian of $\psi_{\gk}$ is not bounded near $\x$, nor  for the version \eqref{Dgk1} since here $r\ge 1$. Hence, a new technique is required. From \eqref{mpim1}, we have \begin{eqnarray} \label{mp11}
&&\hspace{-0.4cm} \int_0^T \|\dot{p}_{\gk}(t)\|\sp dt  \leq  \\ &&\hspace{-0.4cm}  \nonumber  \int_0^T \Bigg\|\underbrace{\left(- (1-\beta)\zeta_{\gk}(t)\tran- \beta\sp \hzeta_{\gk}(t)\tran+\theta_{\gk}(t)+\sum_{i=1}^r\gk e^{\gk \psi_i(x_{\gk}(t))}\; \vartheta^i_{\gk}(t)\right)}_{\mathcal{B}_{\gk}(t)}p_{\gk}(t)\Bigg\| dt \\ &+& \nonumber \bp \sum_{i=1}^r \int_0^T \left\| \gk^2 e^{\gk \psi_i(x_{\gk}(t))}\nabla\psi_i(x_{\gk}(t))\; \<\nabla\psi_i(x_{\gk}(t)),p_{\gk}(t)\>\right\| dt+\l_{\gk}\int_0^T\|\omega_{\gk}(t)\|\sp dt \\ &\leq& \nonumber \underbrace{\big[T(M_\ell+K)+2TM_{\psi}\sqrt{M_{\xi}}\sp\sp \big]}_{\bar{M}_2}\|p_{\gk}\|_\infty +Td_{\gk} \\&+&  \nonumber \underbrace{\sum_{i=1}^r\int_0^T \gk^2 e^{\gk \psi_i(x_{\gk}(t))}\left\| \nabla\psi_i(x_{\gk}(t))\right\|\,\left|\<\nabla\psi_i(x_{\gk}(t)),p_{\gk}(t)\>\right|\sp dt}_{\bf I}\\\nonumber &\leq& M_2(\|p_{\gk}\|_{\infty}+d_{\gk})+{\bf I}, \end{eqnarray}
where $M_2:=\max\{\bar{M}_2,T\}$. In order to prove that {\bf I} is uniformly bounded, we write ${\bf I}= {\bf I}_1+{\bf I}_2,$ in which  $\bar{a}$  is the positive constant in Lemma \ref{A2.2lemma}:
\begin{itemize}
\item $\displaystyle {\bf I}_1:=\sum_{i=1}^r\int_{[I^{\bar{a}}(\x)]^c} \gk^2 e^{\gk \psi_i(x_{\gk}(t))}\left\| \nabla\psi_i(x_{\gk}(t))\right\|\,\left|\<\nabla\psi_i(x_{\gk}(t)),p_{\gk}(t)\>\right|\sp dt.$  
\item  $\displaystyle {\bf I}_2:=\sum_{i=1}^r\int_{I^{\bar{a}}(\x)} \gk^2 e^{\gk \psi_i(x_{\gk}(t))}\left\| \nabla\psi_i(x_{\gk}(t))\right\|\,\left|\<\nabla\psi_i(x_{\gk}(t)),p_{\gk}(t)\>\right|\sp dt$.
\end{itemize}
Since $x_{\gk}$ converges uniformly to $\x$, then,   assumption (A2.1)  and  $x_{\gk}(t)\in C$ for all $t\in [0,T]$, imply the existence of $\bar{k}\in\N$ such that for $k\geq \bar{k}$, 
\begin{equation}\label{pmpeq9} |\psi_i(x_{\gk}(t))-\psi_i(\x(t))|\leq\frac{\ba}{2}\;\;\hbox{for all}\;t\in [0,T]\;\hbox{and for}\;i=1,\dots,r.\end{equation}
Hence, using the definition of $I^{\bar{a}}(\x)$ in \eqref{Iax}, it follows that,  for $k\geq \bar{k}$,
\begin{equation}\label{importantchadi1} \psi_i(x_{\gk}(t))\leq-\frac{\bar{a}}{2}, \;\hbox{for all}\;i=1,\dots,r\;\,\hbox{and}\;\hbox{for all}\;t\in  [I^{\bar{a}}(\x)]^c.
\end{equation}
Thus, there exists $\bar{M}_3>0$ and $\bar{k}_o\geq \bar{k}$ such that for $k\geq \bar{k}_o$ and for $i=1,\dots,r$, we have 
\begin{equation*}\label{I1boundedbis}\hspace{-.1 in} \int_{[I^{\bar{a}}(\x)]^c} \gk^2 e^{\gk \psi_i(x_{\gk}(t))}\left\| \nabla\psi_i(x_{\gk}(t))\right\|\,\left|\<\nabla\psi_i(x_{\gk}(t)),p_{\gk}(t)\>\right|\sp dt \leq  \bar{M}_3\|p_{\gk}\|_{\infty}.\end{equation*} 
This yields that for $k\geq \bar{k}_o$, 
\begin{equation}\label{I1bounded} {\bf I}_1 \leq r\bar{M}_3\|p_{\gk}\|_{\infty}:=M_3\|p_{\gk}\|_{\infty}.\end{equation}
Next, the definition of $I^{\bar{a}}(\x)$ and $\I^{\ba}_{\x(t)}$ given in \eqref{Iax},  and equation \eqref{pmpeq9} yield that, for $k\geq \bar{k}_o$, 
\begin{equation}\label{pmpeq7} \psi_i(x_{\gk}(t))\leq -\frac{\ba}{2}\;\;\hbox{for all}\;t\in I^{\bar{a}}(\x)\;\hbox{and for all}\;i\in[\I^{\ba}_{\x(t)}]^c. \end{equation}
Hence, for $k\geq \bar{k}_o$, \begin{eqnarray*}{\bf I}_2&=& \int_{I^{\bar{a}}(\x)} \sum_{i=1}^r\ \gk^2 e^{\gk \psi_i(x_{\gk}(t))}\left\| \nabla\psi_i(x_{\gk}(t))\right\|\,\left|\<\nabla\psi_i(x_{\gk}(t)),p_{\gk}(t)\>\right|\sp dt \\&= 
& \int_{I^{\bar{a}}(\x)} \sum_{i\in\I^{\bar{a}}_{\x(t)}}\ \gk^2 e^{\gk \psi_i(x_{\gk}(t))}\left\| \nabla\psi_i(x_{\gk}(t))\right\|\,\left|\<\nabla\psi_i(x_{\gk}(t)),p_{\gk}(t)\>\right|\sp dt \\&+& \int_{I^{\bar{a}}(\x)} \sum_{i\in[\I^{\bar{a}}_{\x(t)}]^c}\ \gk^2 e^{\gk \psi_i(x_{\gk}(t))}\left\| \nabla\psi_i(x_{\gk}(t))\right\|\,\left|\<\nabla\psi_i(x_{\gk}(t)),p_{\gk}(t)\>\right|\sp dt\\&\leq& \int_{I^{\bar{a}}(\x)} \sum_{i\in\I^{\bar{a}}_{\x(t)}}\ \gk^2 e^{\gk \psi_i(x_{\gk}(t))}\left\| \nabla\psi_i(x_{\gk}(t))\right\|\,\left|\<\nabla\psi_i(x_{\gk}(t)),p_{\gk}(t)\>\right|\sp dt \\&+& \int_{I^{\bar{a}}(\x)} \sum_{i\in[\I^{\bar{a}}_{\x(t)}]^c}\ \gk^2 \bar{M}_{\psi}^2 e^{-\gk(\frac{\bar{a}}{2})}\|p_{\gk}\|_{\infty}\sp dt. \end{eqnarray*}
This yields the existence of $M_4>0$ and $\bar{k}_1\geq \bar{k}_o$ such that for $ k\geq \bar{k}_1$,
\begin{equation} {\bf I}_2\leq  \int_{I^{\bar{a}}(\x)} \sum_{i\in\I^{\bar{a}}_{\x(t)}}\ \gk^2 e^{\gk \psi_i(x_{\gk}(t))}\left\| \nabla\psi_i(x_{\gk}(t))\right\|\,\left|\<\nabla\psi_i(x_{\gk}(t)),p_{\gk}(t)\>\right|\sp dt  +  {M}_4\|p_{\gk}\|_{\infty}.\label{newp1}\end{equation}
Now let $i\in\{1,\dots,r\}$. Since $\nabla\psi_i$ is Lipschitz, and $x_{\gk}$ and $p_{\gk}$ are absolutely continuous, we deduce that  the function $\left|\<p_{\gk}(\cdot),\nabla\psi_i(x_{\gk}(\cdot))\>\right|$ is absolutely continuous. Thus,  for $t\in [0,T]$ a.e., there exists $\bar\vartheta_{\gk}^i(t)\in \partial^2\psi_i(x_{\gk}(t))$ such that 
\begin{eqnarray}\label{mp10} && \frac{d}{dt}\left|\<p_{\gk}(t),\nabla\psi_i(x_{\gk}(t))\>\right|\\ &=& \nonumber \left[\<\dot{p}_{\gk}(t),\nabla\psi_i(x_{\gk}(t))\>+\<	\bar\vartheta^i_{\gk}(t)\dot{x}_{\gk}(t),p_{\gk}(t)\>\right]\underbrace{\sign (\<p_{\gk}(t),\nabla\psi_i(x_{\gk}(t))\>)}_{s^i_{\gk}(t)}.
\end{eqnarray}
Now substitute into (\ref{mp10}) the term  $\dot{p}_{\gk}(t)$ from \eqref{mpim1}, we obtain that for $t\in [0,T]$ a.e.
\begin{eqnarray} \label{ai}  &&\sum_{j=1}^r s_{\gk}^i(t) \gk^2 e^{\gk \psi_j(x_{\gk}(t))}\<\nabla\psi_j(x_{\gk}(t)),\nabla\psi_i(x_{\gk}(t))\>\<\nabla\psi_j(x_{\gk}(t)),p_{\gk}(t)\> \\\nonumber  &=&\nonumber \frac{d}{dt}|\<p_{\gk}(t),\nabla\psi_i(x_{\gk}(t))\>|-s_{\gk}^i(t)\<\bar\vartheta^i_{\gk}(t)\dot{x}_{\gk}(t),p_{\gk}(t)\>\\ \nonumber &-&  s_{\gk}^i(t)\left\<\B_{\gk}(t)p_{\gk}(t) +\l_{\gk} \omega_{\gk}(t),\nabla\psi_i(x_{\gk}(t))\right\>.
\end{eqnarray}
As $x_{\gk}$ converges uniformly to $\x$, there exists $\bar{k}_2\geq \bar{k}_1$ such that for $k\geq \bar{k}_2$, $x_{\gk}(t)\in \bar{B}_{\bar{\rho}}(\x(t))$ for all $t\in [0,T]$, where $\bar{\rho}$ is the constant of Lemma \ref{A2.2lemma}. This latter implies that,  for $t\in I^{\bar{a}}(\x)$ and $k\geq \bar{k}_2$,  we have
\begin{eqnarray}\label{sum} && \sum_{i\in\I^{\bar{a}}_{\x(t)}} \sum_{j=1}^r s_{\gk}^i(t) \gk^2 e^{\gk \psi_j(x_{\gk}(t))}\<\nabla\psi_j(x_{\gk}(t)),\nabla\psi_i(x_{\gk}(t))\>\<\nabla\psi_j(x_{\gk}(t)),p_{\gk}(t)\> \\ \nonumber  &=&  \sum_{j=1}^r  \sum_{i\in\I^{\bar{a}}_{\x(t)}}  s_{\gk}^i(t) \gk^2 e^{\gk \psi_j(x_{\gk}(t))}\<\nabla\psi_j(x_{\gk}(t)),\nabla\psi_i(x_{\gk}(t))\>\<\nabla\psi_j(x_{\gk}(t)),p_{\gk}(t)\> \\\nonumber &=&\bp\bp  \sum_{j\in\I^{\bar{a}}_{\x(t)}}  \sum_{i\in\I^{\bar{a}}_{\x(t)}}  s_{\gk}^i(t) \gk^2 e^{\gk \psi_j(x_{\gk}(t))}\<\nabla\psi_j(x_{\gk}(t)),\nabla\psi_i(x_{\gk}(t))\>\<\nabla\psi_j(x_{\gk}(t)),p_{\gk}(t)\> \\ \nonumber  &+& \bp\bp\bp \bp\sum_{j\in[\I^{\bar{a}}_{\x(t)}]^c}  \sum_{i\in\I^{\bar{a}}_{\x(t)}}  s_{\gk}^i(t) \gk^2 e^{\gk \psi_j(x_{\gk}(t))}\<\nabla\psi_j(x_{\gk}(t)),\nabla\psi_i(x_{\gk}(t))\>\<\nabla\psi_j(x_{\gk}(t)),p_{\gk}(t)\> \\\nonumber &=&  \bp\bp  \sum_{j\in\I^{\bar{a}}_{\x(t)}}  \Bigg\{ \gk^2 e^{\gk \psi_j(x_{\gk}(t))} \Bigg(\|\nabla\psi_j(x_{\gk}(t))\|^2\\\nonumber &+&  \sum_{\substack{i\in\I^{\bar{a}}_{\x(t)}\\ i\not=j}} s_{\gk}^i(t) s_{\gk}^j(t) \<\nabla\psi_j(x_{\gk}(t)),\nabla\psi_i(x_{\gk}(t))\> \Bigg) |\<\nabla\psi_j(x_{\gk}(t)),p_{\gk}(t)\>|\Bigg\}\\ \nonumber &+&   \bp\bp\bp \bp\sum_{j\in[\I^{\bar{a}}_{\x(t)}]^c}  \sum_{i\in\I^{\bar{a}}_{\x(t)}}  s_{\gk}^i(t) \gk^2 e^{\gk \psi_j(x_{\gk}(t))}\<\nabla\psi_j(x_{\gk}(t)),\nabla\psi_i(x_{\gk}(t))\>\<\nabla\psi_j(x_{\gk}(t)),p_{\gk}(t)\> \\\nonumber   &\geq& (1-\bar{b})  \sum_{j\in\I^{\bar{a}}_{\x(t)}}  \gk^2 e^{\gk \psi_j(x_{\gk}(t))} \|\nabla\psi_j(x_{\gk}(t))\|^2  |\<\nabla\psi_j(x_{\gk}(t)),p_{\gk}(t)\>|\\\nonumber &+&  \bp\bp\bp \bp\sum_{j\in[\I^{\bar{a}}_{\x(t)}]^c}  \sum_{i\in\I^{\bar{a}}_{\x(t)}}  s_{\gk}^i(t) \gk^2 e^{\gk \psi_j(x_{\gk}(t))}\<\nabla\psi_j(x_{\gk}(t)),\nabla\psi_i(x_{\gk}(t))\>\<\nabla\psi_j(x_{\gk}(t)),p_{\gk}(t)\>.\end{eqnarray}
Combining this latter inequality for \eqref{sum} with  the summation  over $i\in \I^{\bar{a}}_{\x(t)}$ of \eqref{ai},  and using \eqref{pmpeq7},  we deduce that for $t\in I^{\bar{a}}(\x)$ and for $k\geq \bar{k}_2$, 
\begin{eqnarray}& &\label{finalproof1} \sum_{j\in\I^{\bar{a}}_{\x(t)}} \gk^2 e^{\gk \psi_j(x_{\gk}(t))} \|\nabla\psi_j(x_{\gk}(t))\|^2  |\<\nabla\psi_j(x_{\gk}(t)),p_{\gk}(t)\>|\\\nonumber &\leq& \frac{1}{1-\bar{b}} \Bigg\{ \sum_{i\in\I^{\bar{a}}_{\x(t)}} \bigg[ \frac{d}{dt}|\<p_{\gk}(t),\nabla\psi_i(x_{\gk}(t))\>|-s_{\gk}^i(t)\<\bar\vartheta^i_{\gk}(t)\dot{x}_{\gk}(t),p_{\gk}(t)\>\\\nonumber  &-&  s_{\gk}^i(t)\left\<\B_{\gk}(t)p_{\gk}(t) +\l_{\gk} \omega_{\gk}(t),\nabla\psi_i(x_{\gk}(t))\right\>\bigg] + \sum_{j\in[\I^{\bar{a}}_{\x(t)}]^c}  \sum_{i\in\I^{\bar{a}}_{\x(t)}} \bar{M}_{\psi}^{3}\gk^2 e^{-\frac{1}{2}\gk \bar{a}} \|p_{\gk}\|_{\infty}\Bigg\}. \end{eqnarray}
From \eqref{ai}, we have for $i=1,\dots, r$, that
\begin{eqnarray*} \label{totalderiv}  && \frac{d}{dt}|\<p_{\gk}(t),\nabla\psi_i(x_{\gk}(t))\>|=\\\nonumber && + \sum_{j=1}^r s_{\gk}^i(t) \gk^2 e^{\gk \psi_j(x_{\gk}(t))}\<\nabla\psi_j(x_{\gk}(t)),\nabla\psi_i(x_{\gk}(t))\>\<\nabla\psi_j(x_{\gk}(t)),p_{\gk}(t)\> \\\nonumber  &+&\nonumber s_{\gk}^i(t)\<\bar\vartheta^i_{\gk}(t)\dot{x}_{\gk}(t),p_{\gk}(t)\>\\ \nonumber &+&  s_{\gk}^i(t)\left\<\B_{\gk}(t)p_{\gk}(t) +\l_{\gk} \omega_{\gk}(t),\nabla\psi_i(x_{\gk}(t))\right\>.
\end{eqnarray*}
This gives using \eqref{importantchadi1}, \eqref{pmpeq7} and the uniform boundedness of $(\dot{x}_{\gk})_k$, the existence of $M_5>0$ and $\bar{k}_3\geq \bar{k}_2$ such that for $k\geq \bar{k}_3$, we have
\begin{equation} \label{totalderivbis} \left| \int_{[I^{\bar{a}}(\x)]^c}\sum_{i=1}^r\frac{d}{dt}\left|\<p_{\gk}(t),\nabla\psi_i(x_{\gk}(t))\>\right|\sp dt\right|\leq M_5(\|p_{\gk}\|_\infty+{d}_{\gk}),\;\hbox{and}\end{equation}
\begin{equation} \label{totalderivbisbis}\left| \int_{I^{\bar{a}}(\x)}\sum_{i\in [\I^{\bar{a}}_{\x(t)}]^c}\frac{d}{dt}\left|\<p_{\gk}(t),\nabla\psi_i(x_{\gk}(t))\>\right|\sp dt\right|\leq M_5(\|p_{\gk}\|_\infty+{d}_{\gk}).\end{equation}
Add to \eqref{totalderivbis} that $$\left|\int_0^T\sum_{i=1}^r\frac{d}{dt}\left|\<p_{\gk}(t),\nabla\psi_i(x_{\gk}(t))\>\right|\sp dt\right| \leq 2r\bar{M}_\psi \|p_{\gk}\|_{\infty}, $$
we deduce the existence of $M_6>0$ and $\bar{k}_4\geq \bar{k}_3$ such that for $k\geq \bar{k}_4$, we have
$$ \left| \int_{I^{\bar{a}}(\x)}\sum_{i=1}^r\frac{d}{dt}\left|\<p_{\gk}(t),\nabla\psi_i(x_{\gk}(t))\>\right|\sp dt\right|\leq  M_6(\|p_{\gk}\|_\infty+d_{\gk}).$$
This latter inequality with \eqref{totalderivbisbis} yield the existence of $M_7>0$ and $\bar{k}_5\geq \bar{k}_4$ such that for $k\geq \bar{k}_5$, we have
\begin{equation} \label{totalderivbisbis1}\left| \int_{I^{\bar{a}}(\x)}\sum_{i\in \I^{\bar{a}}_{\x(t)}}\frac{d}{dt}\left|\<p_{\gk}(t),\nabla\psi_i(x_{\gk}(t))\>\right|\sp dt\right|\leq M_7(\|p_{\gk}\|_\infty+d_{\gk}).\end{equation}
Now, integration the both sides of \eqref{finalproof1} on $I^{\bar{a}}(\x)$, and using \eqref{totalderivbisbis1}, we get the existence of $M_8>0$ and and $\bar{k}_6\geq \bar{k}_5$ such that for $ k\geq \bar{k}_6$, \begin{equation}\label{newimport}  \int_{I^{\bar{a}}(\x)} \sum_{j\in\I^{\bar{a}}_{\x(t)}} \gk^2 e^{\gk \psi_j(x_{\gk}(t))} \|\nabla\psi_j(x_{\gk}(t))\|^2  |\<\nabla\psi_j(x_{\gk}(t)),p_{\gk}(t)\>|\sp dt \leq M_8(\|p_{\gk}\|_\infty+d_{\gk}).\end{equation}
Using that $x_{\gk}$ converges uniformly to $\x$, and by assuming that $\bar{a}\leq \frac{\e}{2}$, where $\e$ is the constant of  \eqref{pisieps}, we get the existence of $\bar{k}_7\geq \bar{k}_6$ such that for $k\geq \bar{k}_7$, $t\in I^{\bar{a}}(\x)$ and $j\in\I^{\bar{a}}_{\x(t)}$, we have $\psi_j(x_{\gk}(t))\geq -\e$, and hence by \eqref{pisieps}, $\|\nabla\psi_j(x_{\gk}(t)\|>\eta.$
Then, for $M_9:=\frac{M_8}{\eta}$, \eqref{newimport} yields that for $k\geq \bar{k}_7$, \begin{equation*}\label{pmpeq10bis}\int_{I^{\bar{a}}(\x)} \sum_{j\in\I^{\bar{a}}_{\x(t)}} \gk^2 e^{\gk \psi_j(x_{\gk}(t))} \|\nabla\psi_j(x_{\gk}(t))\|  |\<\nabla\psi_j(x_{\gk}(t)),p_{\gk}(t)\>|\sp dt \leq M_9(\|p_{\gk}\|_\infty+d_{\gk}). \end{equation*}
Combining this latter with \eqref{newp1}, we conclude that for $k\geq \bar{k}_7$, ${\bf I}_2\leq M_9(\|p_{\gk}\|_{\infty}+d_{\gk})+M_4\|p_{\gk}\|_\infty$, and hence by \eqref{I1bounded} and for $M_{10}:=M_3+M_4+M_9,$
\begin{equation}\label{pmpeq10}{\bf I}={\bf I}_1+{\bf I}_2\leq M_{10}(\|p_{\gk}\|_{\infty} +d_{\gk}).\;\;\end{equation}
 Therefore, for $k\geq\bar{k}_7$, we have from \eqref{ML}, \eqref{mp7}, \eqref{mp11} and \eqref{pmpeq10}, that $$ \int_0^T \|\dot{p}_{\gk}(t)\|\sp dt \leq (M_2+M_{10})(\|p_{\gk}\|_{\infty} +d_{\gk})\leq  (M_2+M_{10})(M_1 +M_L).$$
 This yields that $(\dot{p}_{\gk})$ is uniformly bounded in $L^1$, which terminates the proof of Theorem \ref{thm1mpcomp} under the temporary assumption: $C$ is compact.\vspace{-0.2cm}\\
  
 Now, we proceed to show that the temporary assumption \enquote{\it $C$ is compact} can be replaced by (A2.4).  Let $C$ be unbounded, but satisfies (A2.4) for some $y_o$ and $R_o$.  We introduce an additional constraint to $C$ via  $\psi_{r+1}\colon\R^n\f\R,$  where  $$\psi_{r+1}(x):=\frac{1}{2}\left(\|x-y_o\|^2-R_o^2\right)\;\;\hbox{for all}\; x\in
\R^n.$$ 
Hence, $C_{r+1}$, the zero sublevel set of $\psi_{r+1}$, is $\bar{B}_{R_o}(y_o)$ and it contains $\x(t)$, for all $t\in [0,T]$,  in its interior. Denote by  $(\hat{P})$  the problem $(P)$  with the following modifications: 
\begin{itemize}
\item The {\it unbounded} sweeping set $C$ is now replaced by the {\it compact} set $$\hat{C}:=C\cap C_{r+1}=C\cap\bar{B}_{R_o}(y_o) =\bigcap_{i=1}^{r+1}\{x\in\R^n : \psi_i(x)\leq 0\}.$$
\item The function $\varphi$ is replaced by the function $\hat{\varphi}:=\varphi + I_{{C_{r+1}}}$, where $I_{{C_{r+1}}}$ is the indicator function of the set $C_{r+1}$. Clearly we have $\dom \hat{\varphi}=\hat{C}$ and (A3)  is satisfied  $\hat{\varphi}$ on $\hat{C}$. Moreover, since $\varphi= \hat{\varphi}$ on the open set $\inte C_{r+1}$, we have \begin{equation}\label{partialphihat} \partial \varphi(x)=\partial\hat{\varphi}(x),\;\;\forall x\in C\cap\inte{C_{r+1}}. \end{equation}
\item The set $C_0$ is replaced by the closed set \begin{equation} \label{C0hat} \hat{C}_0:=C_0\cap C_{r+1}=C_0\cap \bar{B}_{R_o}(y_o)\subset C\cap \bar{B}_{R_o}(y_o)=\hat{C}=\dom \hat{\varphi}.\end{equation}
\end{itemize}
We claim that $(\x,\u)$ is a strong local minimizer for $(\hat{P})$. Indeed, since $(\x,\u)$ is admissible for $(P)$ and $\x(t)\in C\cap \inte C_{r+1}$ for all $t\in [0,T]$, then \eqref{partialphihat} and \eqref{C0hat} yield that  $(\x,\u)$ is admissible for $(\hat{P})$. For the optimality, let $\hat{\delta}>0$ satisfying \begin{equation}\label{deltahat}\hat{\delta}<\min\{\delta,R_o-\bar{\e}\},\;\;\hbox{where}\;\bar{\e}:=\max\{\|\x(t)-y_o\|,\;t\in[0,T]\}<R_o.\end{equation}
Let $(x,u)$ be admissible for $(\hat{P})$ such that $\|x- \x\|_{\infty}\leq\hat{\delta}$. Then, $x(t)\in \hat{C}$ for all $t\in [0,T]$, and by \eqref{deltahat},  $\|x- \x\|_{\infty}\leq{\delta}$  and $\|x(t)-y_o\|<R_o$ for all $t\in [0,T]$. Hence, \eqref{partialphihat} implies that $(x,u)$ is admissible for $(P)$. Thus, the optimality of $(\x,\u)$ for $(P)$ yields that $g(\x(0),\x(T))\leq g(x(0),x(T))$, which shows   the optimality of $(\x,\u)$ for $(\hat{P})$. Now, since $\x(t)\in \inte C_{r+1}$ for all $t\in [0,1]$, we have: \begin{itemize}
\item $\hat{I}^0(\x):=\{t\in [0,T] : \x(t)\in\bdry \hat{C}\}=\{t\in [0,T] : \x(t)\in\bdry C\}=I^0{(\x)}.$
\item For $t\in \hat{I}^0(\x)$, \begin{equation*} \hat{\I}^0_{\x(t)}: =\{i\in\{1,\dots,r+1\} : \psi_i(\x(t))=0\}=\{i\in\{1,\dots,r\} : \psi_i(\x(t))=0\}=\I^0_{\x(t)}. \end{equation*}
\end{itemize}
This yields that  (A2.3)  is the same for the family  $(\nabla\psi_i(\x(\cdot))_{i=1}^{r+1}$. However,  having  (A2.2)  satisfied by the family  $(\nabla\psi_i)_{i=1}^{r+1}$, that now includes $\nabla\psi_{r+1}$,  is not automatic, but it holds true   under assumption (A2.4)$(i)$, due to  Lemma \ref{A2.2H2}. 
 On the other hand, since $\hat{C}\subset C$, $\hat{C}_0\subset C_0$ and $\inte \hat{C}=\inte C\cap B_{R_o}(y_o)$, we conclude that the data of the problem $(\hat{P})$ satisfies  (A2.5), (A5) and (A6), and thus, satisfies all the assumptions  (A1)-(A2.3) and (A2.5)-(A6), with $\hat{C}$ compact. Therefore, the proof of this theorem, where (A2.4) holds,  is completed by applying  to the strong local minimizer $(\x,\u)$ of $(\hat{P})$  the  version of this theorem already proven for $C$ {\it compact},  and by noticing the following:
\begin{itemize}
\item  $\hat{C}\subset C$ yields  the set $\hat{C}$ used in the definitions of $\partial_\ell\varphi$, $\partial^2_\ell\varphi$, $\partial^{\sp x}_\ell f (t,\cdot,u)$, $\partial^2_\ell\psi_i$ and $\partial^L_\ell g$, can be replaced by $C$. 
\item $\x(t)\in \inte C_{r+1}$ for all $t\in [0,1]$, implies that $\xi^{r+1}\equiv 0$, $\vartheta_{r+1}\equiv 0$ and $\nu_{r+1}\equiv 0$.
\item The local property of the limiting normal cone and   $\x(0)\in\inte \bar{B}_{R_o}(x_o)$ (by (A2.4)$(ii)$),  give that   $$N_{\hat{C}_0}^L(\x(0))= N_{C_0\cap \bar{B}_{R_o}(x_o)}^L(\x(0))=N_{{C}_0}^L(\x(0)).$$
\end{itemize}

For the ``Furthermore" part of the theorem, let $C_T=\R^n$.  Proposition \ref{prop4old} implies that  $\l =1$. The convexity assumption of $f(t,x,U)$ for $x\in C \cap \bar{B}_{\delta}(\x(t))$ and $t\in [0,T]$, is removed  by using the {\it relaxation} technique of \cite[Section 5.2]{verachadi}  in the same fashion as in Step 7 of the proof of \cite[Theorem 5.1]{verachadi}. 

The proof of Theorem \ref{thm1mpcomp} is terminated.\hspace{0.3cm}  \tallqed
 
 \begin{remark} In the nontriviality condition of Theorem \ref{thm1mpcomp}, the presence of  $\|p(T)\|$ instead of $\|p\|_{\infty}$ and $(\|\nu^i\|_{\TV})_{i=1}^r$ results from having these two norms bounded above by $\|p(T)\|$. Indeed, if we take $k\f\infty$ in \eqref{mp7}, we conclude that \begin{equation*}\label{normofp} \|p\|_\infty \leq e^{(M_\ell+K+2M_{\psi}\sqrt{M_\xi})}\|p(T)\|=\frac{M_1}{1+L_g+M_L}\|p(T)\|.\end{equation*} 
On the other hand, using  \eqref{pisieps} and \eqref{pmpeq10}, and by applying on each $\psi_i$ the same argument employed in \cite[Equation (82)]{verachadi} we deduce  the existence of $M_{10}>0$ and $\bar{k}_9\geq \bar{k}_8$ such that for $k\geq \bar{k}_9$ and for $i=1,\dots,r$, we have \begin{equation*}\label{pmpeq10bisbis} \int_0^T \gk^2 e^{\gk \psi_i(x_{\gk}(t))}\left|\<\nabla\psi_i(x_{\gk}(t)),p_{\gk}(t)\>\right|\sp dt \leq M_{10}(\|p_{\gk}\|_{\infty}+ d_{\gk}).\end{equation*}
This gives, using \eqref{mp7}, that for $k\geq \bar{k}_9$ and for $i=1,\dots,r,$ we have \begin{equation}\label{radonbound} \|\nu^i_{\gk}\|_{\TV} \leq M_{10}(\|p_{\gk}\|_{\infty}+d_{\gk})\leq\frac{M_{10}M_1}{1+L_g+M_L}(\|p_{\gk}(T)\|+d_{\gk}) + M_{10}d_{\gk},\end{equation}
where $\nu^i_{\gk}$ be the finite signed Radon measure on $[0,T]$ defined by \begin{equation*} \label{dnuigk} d\nu^i_{\gk}(t):= \gk \xi^i_{\gk}(t)\<\nabla\psi_i(x_{\gk}(t)),p_{\gk}(t)\>\sp dt.\end{equation*}
Since for each $i\in\{1,\dots,r\}$, the signed Radon measure $\nu^i$ of Theorem \ref{thm1mpcomp} is the weak* limit of $\nu_{\gk}^i$, see Step 4 of \cite[Proof of Theorem 6.1]{VCpaper} for more details, we deduce after taking $k\f\infty$ in  \eqref{radonbound} that  \begin{equation*}\label{normtvnui}  \|\nu^i\|_{\TV}\leq \frac{M_{10}M_1}{1+L_g+M_L}\|p(T)\|.\end{equation*} 
 \end{remark}

\section{Appendix} \label{auxresults} In this section, we present an example to which our Pontryagin-type maximum principle (Theorem \ref{thm1mpcomp}) is applied to obtain an optimal solution. Furthermore, we establish auxiliary results that are used in different places of the paper, and we provide  proofs for Propositions \ref{proppsigk} and \ref{propcgk(k)}.

\subsection{Example}\label{examplesec}

In this example we illustrate how Theorem \ref{thm1mpcomp}  can be applied to find an optimal solution. We consider the following  data for $(P)$ (see Figure \ref{Fig1}): 
\begin{itemize}[leftmargin=*]
\item The perturbation mapping $f\colon\R^3\times\R\f\R^3$  is defined by $$f((x_1,x_2,x_3),u)=(4x_1+u,x_2-1,-2x_1-x_2+u+2).$$
\item The two functions $\psi_1,\;\psi_2\colon \R^3\f\R$ are defined by $$\psi_1(x_1,x_2,x_3):=x_1^2+x_2^2+x_3\,\;\hbox{and}\,\;\psi_2(x_1,x_2,x_3):=x_1^2+(x_2-2)^2+x_3,$$ and hence, the set $C$ is the {\it nonsmooth}, convex and {\it  unbounded} set $$C=C_1\cap C_2:=\{(x_1,x_2,x_3) : \psi_1(x_1,x_2,x_3)\leq 0\}\cap \{(x_1,x_2,x_3) : \psi_2(x_1,x_2,x_3)\leq 0\}.$$
\item The objective function $g\colon\R^6\f\R\cup\{\infty\}$  is defined  by $$g(x_1,x_2,x_3,x_4,x_5,x_6):=\begin{cases} -x_4^2-x_6-1& \;\;(x_4,x_5,x_6)\in C, \vspace{0.1cm}\\ \infty&\;\;\hbox{Otherwise}. \end{cases}$$
\item The function $\varphi$ is the indicator function of $C$ and $T:=\tfrac{1}{2}.$ 
\item \sloppy The control multifunction is the constant $U(t):=[-1,1]$ for all $t\in [0,\frac{1}{2}]$, $C_0:=\{(0,1,-1)\}$, and $C_T:=\{(x_1,x_2,x_3) \in\R^3: 8x_1-4x_3 -9=0\}.$
\end{itemize}
\begin{figure}[tb]
\centering
\includegraphics[scale=0.45]{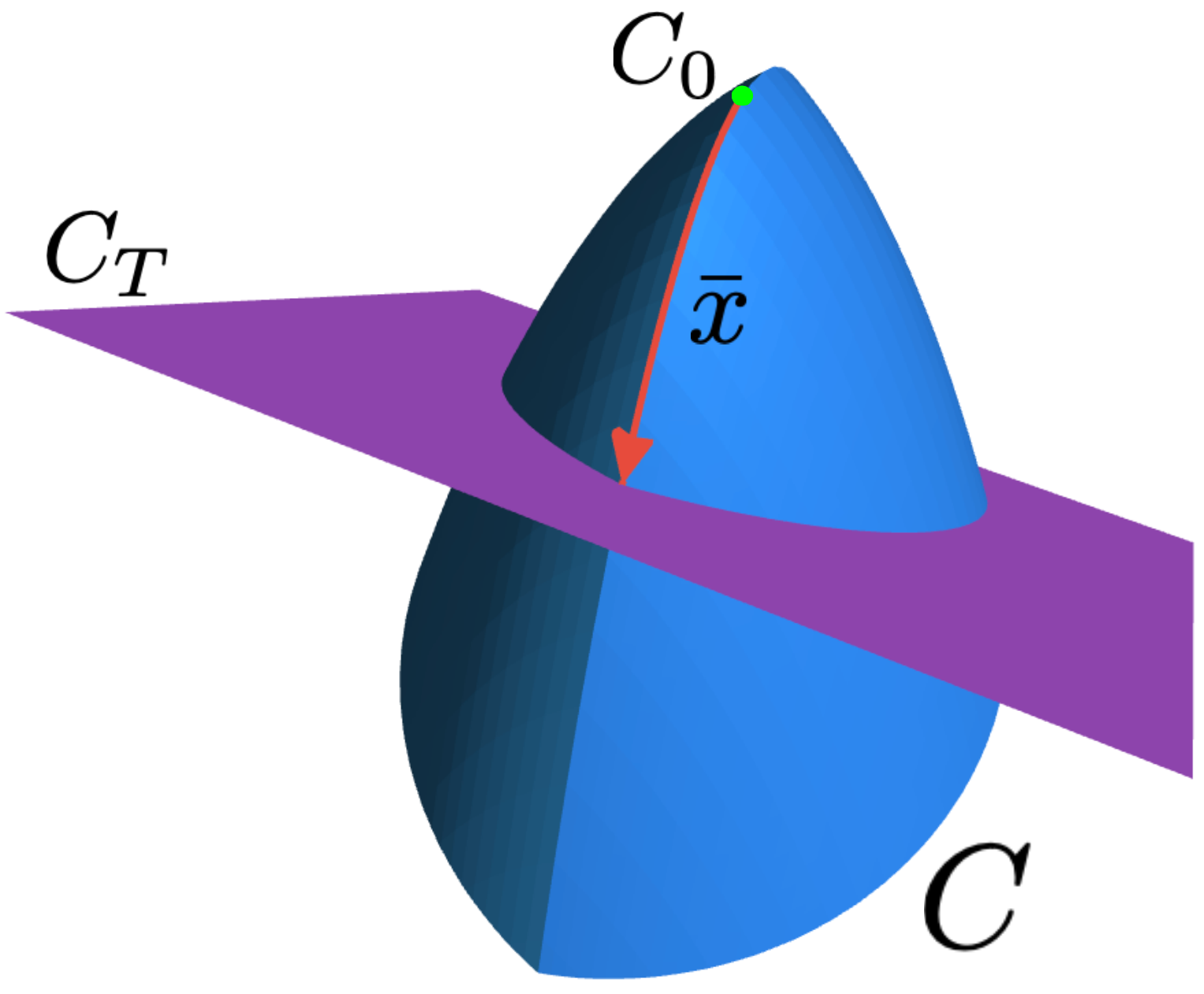}
\caption{\label{Fig1} Example \ref{examplesec}}
\end{figure} 
One can easily verify that (A1)-(A2.2) and (A2.4)-(A6) are satisfied. Define   the curve
 $$\Gamma:=\{(x_1,x_2,x_3) : x_1^2+x_3+1=0\;\hbox{and}\;x_2=1\}= (\bdry C_1\cap \bdry C_2)\subset\bdry C.$$ 
 Since $C_0\subset \Gamma$ and  $g$ vanishes on $\Gamma$   and is strictly positive elsewhere in $C$, we may seek  for $(P)$  a candidate $(\x,\u)$ for  optimality  with  $\x:=(\x_1,\x_2,\x_3)$ belonging to $\Gamma$, if possible,  and hence we have \begin{equation} \label{ex1}\begin{cases} \x_1^2(t)+\x_3(t)+1=0\;\,\hbox{and}\,\;\x_2(t)=1\;\forall t\in [0,\tfrac{1}{2}]\,\;\hbox{and}\\[1pt] \x(0)\tran=(0,1,-1)\,\;\hbox{and}\;\,\x(\tfrac{1}{2})\tran\in \{ (\tfrac{1}{2},1,-\tfrac{5}{4}), (-\tfrac{5}{2},1,-\tfrac{29}{4})\}.\end{cases} \end{equation}
Note that (A2.3) is satisfied on $\Gamma$ for $b=\tfrac{3}{5}$.\footnote{Note that for $(x_1,x_2,x_3)\in\Gamma$ with $-\tfrac{\sqrt{3}}{2} < x_1< \tfrac{\sqrt{3}}{2}$, we have $\<\nabla\psi_1(x_1,x_2,x_3),\nabla\psi_2(x_1,x_2,x_3)\>=4x_1^2-3<0,$ and hence, the maximum principle of \cite{pinho22} cannot be applied to this sweeping set $C$.} Then, applying  Theorem  \ref{thm1mpcomp}  to such candidate $(\x,\u)$ we obtain the existence of an adjoint vector $p:=(p_1,p_2,p_3)\in BV([0,\frac{1}{2}];\R^3)$, two finite signed Radon measures $\nu_1$, $\nu_2$ on $\left[0,\frac{1}{2}\right]$, $\xi_1$, $\xi_2\in L^{\infty}([0,\tfrac{1}{2}];\R^{+})$, and $\l\geq 0$,  such that  when incorporating  equations \eqref{ex1}  into Theorem  \ref{thm1mpcomp}$(i)$-$(vi)$, we obtain 
\begin{enumerate}[(a)]
\item $\|p(\frac{1}{2})\|+\l=1.$
\item The admissibility equation holds, that is, for $t\in[0,\tfrac{1}{2}]$ a.e.,  \begin{equation*}\begin{cases}\dot{\bar{x}}_1(t)= 4\x_1(t)+\u(t)-2\x_1(t)(\xi_1(t)+\xi_2(t)),\\ 
0= -2(\xi_1(t)-\xi_2(t)),\\ 
-2\dot{\x}_1(t)\x_1(t)=-2\x_1(t)+\u(t)+1-((\xi_1(t)+\xi_2(t)).\end{cases} \end{equation*} 
\item The adjoint equation is satisfied, that is, for  $t\in [0,\frac{1}{2}]$, 
\begin{eqnarray*} dp(t)&=&\begin{pmatrix*}\begin{array}{rrr}
\bp -4\, & 0\;\;\, & 2 \\
\bp 0\, & -1\;\;\, & 1 \\
\bp 0\, & 0\;\;\,  & 0 
\end{array}\end{pmatrix*}p(t)\sp dt\sp+\sp (\xi_1(t)+\xi_2(t))\begin{pmatrix} 
2\;\; & 0\;\; & 0  \\
0 \;\; & 2\;\; & 0 \\
0\;\; & 0 \;\; & 0
\end{pmatrix}p(t)\sp dt\sp\\[2pt]&+&\sp \begin{pmatrix} 
2\x_1(t) \\
2\\
1
\end{pmatrix} d\nu_1 + \begin{pmatrix} 
2\x_1(t) \\
-2\\
1
\end{pmatrix} d\nu_2.\end{eqnarray*}
\item The complementary slackness condition is valid, that is, for $t\in[0,\tfrac{1}{2}]$ a.e.,  $$\begin{cases}\xi_1(t)(2p_1(t)\x_1(t) + 2p_2(t)+p_3(t))=0, \\ \xi_2(t)(2p_1(t)\x_1(t) - 2p_2(t)+p_3(t))=0.\end{cases}$$
\item The transversality condition holds, that is, $$\begin{cases} -p(\frac{1}{2})\tran\in \l\{-1,0,-1)\}+\{\a(2,0,-1): \a\in\R\}& \hbox{if}\;\; \x(\tfrac{1}{2})\tran=(\tfrac{1}{2},1,-\tfrac{5}{4}),\\-p(\frac{1}{2})\tran\in \l\{(5,0,-1)\}+\{\a(2,0,-1): \a\in\R\}& \hbox{if}\;\; \x(\tfrac{1}{2})\tran=(-\tfrac{5}{2},1,-\tfrac{29}{4}).\end{cases}$$
\item $\max\{u\sp(p_1(t)+ p_3(t)) : u\in [-1,1]\}$ is attained at $\u(t)$ for $t\in[0,\frac{1}{2}]$ a.e.
\end{enumerate}
We temporarily assume that \begin{equation}\label{exam5} p_1(t)+p_3(t)\geq 0,\;\;\forall t\in[0,\tfrac{1}{2}]\;\hbox{a.e.}\end{equation} 
This gives from (f) that $u(t)=1$ for $t\in[0,\frac{1}{2}]$ a.e. Now solving the differential equations of (b) and using \eqref{ex1}, we obtain that \begin{equation*} \label{lastformu} \xi_1(t)=\xi_2(t)=1\;\,\hbox{and}\;\,\x(t)\tran=(t,1,-1-t^2),\;\;\forall t\in[0,\tfrac{1}{2}].\end{equation*}
Hence, from (d), we deduce that $p_2(t)=0$ for $t\in[0,\frac{1}{2}]$ a.e., and \begin{equation}\label{ex4} 2t\sp p_1(t)+p_3(t)=0,\;\,\forall t\in[0,\tfrac{1}{2}]\;\textnormal{a.e.} \end{equation}
Moreover, the adjoint equation (c) simplifies to the following
 \begin{equation}\label{ex3} \begin{cases}d{p}_1(t)= 2p_3(t)\sp dt + 2t d\nu_1+2td\nu_2,\\
0= p_3(t)\sp dt + 2d\nu_1-2d\nu_2,\\
 d{{p}}_3(t)= d\nu_1+ d\nu_2.  \end{cases} \end{equation} 
Using (a), \eqref{ex4}, (e), and \eqref{ex3}, a simple calculation gives that 
$$\begin{cases} \l=\tfrac{1}{4}\;\,\hbox{and}\;\,p(\tfrac{1}{2})\tran=(\tfrac{3}{4},0,0),\\[2pt]p(t)\tran=\left(\frac{3}{4(4t^2+1)},0,\frac{-3t}{2(4t^2+1)}\right)\;\hbox{on}\;[0,\tfrac{1}{2}),\\[6pt] d\nu_1=\tfrac{12t^3+24t^2+3t-6}{8(4t^2+1)^2}\sp dt+\tfrac{3}{16}\sp\delta_{\big\{\bbp\tfrac{1}{2}\bbp\big\}}\;\,\hbox{and}\;\,d\nu_2=\tfrac{-12t^3+24t^2-3t-6}{8(4t^2+1)^2}\sp dt+\tfrac{3}{16}\sp\delta_{\big\{\bbp\tfrac{1}{2}\bbp\big\}} \;\hbox{on}\;[0,\tfrac{1}{2}],\end{cases}\vspace{0.01cm}$$
where $\delta_{\{a\}}$ denotes the unit measure concentrated on the point $a$. Note that for all $t\in [0,\tfrac{1}{2}]$, we have $p_1(t)+p_3(t)\geq 0 $, and hence, the temporary assumption \eqref{exam5} is satisfied.

Therefore, the above analysis, realized via Theorem \ref{thm1mpcomp}, produces  an admissible pair $(\x,\u)$, where $$\textstyle \x(t)\tran=(t,1,-1-t^2)\;\;\hbox{and}\;\;\u(t)=1,\;\;\forall t\in [0,\tfrac{1}{2}],$$
which is optimal for $(P)$.

\subsection{Auxiliary results}

\begin{lemma}   \label{dcexistence} Assume that \textnormal{(A2.1)-(A2.2)} hold and that $C$ is compact. Then for $c\in\bdry C$, there exists a vector $d_c\not=0$ such that \begin{equation*} \label{veradc1} \frac{4\eta^2}{\bar{M}_{\psi}}\leq\|d_c\|\leq r\bar{M}_\psi,\;\,\hbox{and}\;\,\left\<\frac{d_c}{\|d_c\|},\nabla\psi_i(c)\right\>\leq -\frac{4\eta^2}{r\bar{M}_{\psi}},\;\;\forall i\in\I^0_c.\end{equation*} 
\end{lemma}
\begin{proof}  Let $c\in\bdry C$. For each $i\in\I^0_c$, we denote by $u_i(c)$ and $v_i(c)$ the unique projections of $-\nabla\psi_i(c)$ to $N_C(c)$ and $T_C(c)$, respectively.  By Moreau decomposition theorem, see \cite{moreaudecomp}, we have $$-\nabla\psi_i(c)=u_i(c)+v_i(c)\;\hbox{and}\;\<u_i(c),v_i(c)\>=0.$$ This yields that $\|\nabla\psi_i(c)\|^2=\|u_i(c)\|^2+\|v_i(c)\|^2$, and hence \begin{equation}\label{verade1} \|v_i(c)\|\leq \|\nabla\psi_i(c)\|\leq \bar{M}_{\psi}. \end{equation} On the other hand, by Proposition \ref{prop1}$(ii)$, we have that $u_i(c)=\sum_{j\in\I^0_c}\l^i_j\nabla\psi_j(c)$, where $\l_j\geq 0$ for $j\in\I^0_c$. Then by (A2.2), \begin{equation*} \|v_i(c)\|=\|\nabla\psi_i(c)+u_i(c)=\left\|\left(\sum_{j\in\I^0_c,\sp j\not=i}\l^i_j\nabla\psi_j(c)\right) + (1+\l_i^i)\nabla\psi_i(c) \right\|\geq\left(1+\sum_{j\in\I^0_c}\l_j^i\right)2\eta\geq2\eta.\end{equation*} Hence, using \eqref{verade1}, we obtain that \begin{equation}\label{verade2} 2\eta\leq \|v_i(c)\|\leq \bar{M}_{\psi},\;\;\forall i\in \I^0_c.\end{equation} 
It follows that, 
\begin{equation}\label{vipsi}
\left\<  v_i(c), \nabla\psi_i(c)    \right\>= \left\<  v_i(c), -u_i(c)-v_i(c)  \right\>=-\|v_i(c)\|^2\le -4\eta^2.
\end{equation}
We define $d_c:=\sum_{j\in\I^0_c}v_j(c) $. As  $v_j(c)$ belongs to $T_C(c)$  for all $j\in \I^0_c$,  we have that $\left\<v_j(c),\nabla\psi_i(c)\right\> \le 0$, for all $i,j\in \I^0_c$. This, together with  \eqref{vipsi} gives that, for all $i\in \I^0_c$ we have
\begin{equation*}
\bar{M}_{\psi} \|d_c\| \ge  \left\<d_c, -\nabla\psi_i(c)\right\>\ge \left\<v_i(c),-\nabla\psi_i(c)\right\>\ge 4 \eta^2.
\end{equation*}
Whence, using the right inequality in \eqref{verade2}, we have 
\begin{equation*}
 \left\<d_c, \nabla\psi_i(c)\right\>\le - 4\eta^2, \; \forall i\in\I^0_c, \;\;\hbox{and}\; \; \frac{4 \eta^2}{\bar{M}_\psi} \le \|d_c\|\le r\bar{M}_\psi.
 \end{equation*}
This yields that $d_c\neq 0$ and,  upon dividing the first inequality in the above equation  by $\|d_c\|$,  the required result is established. \end{proof}

\begin{lemma}\label{lemmaaux1} Assume that \textnormal{(A2.1)} holds. Let $\a_n\ge 0$,  for all $n\in \N$, with   $\a_n\f\a_o$ and  let  $c_n\in C$ be a sequence such that  $\I^{\a_n}_{c_n}\neq\emptyset$, for all $n\in \N$, and $c_n\f c_o$. Then, $\I^{\a_o}_{c_o}\neq\emptyset$ and there exist $ \emptyset\not=\J_o\subset \{1,\dots,r\}$ and a subsequence of  $(\a_n, c_n)_n$  we do not relabel, such that  
$$\I^{\a_n}_{c_n}= \J_o\subset \I^{\a_o}_{c_o}\;\,\hbox{for all}\;\,n\in\N.$$
In particular, for any continuous function $\x\colon[0,1]\f C$ and for all  $a\ge 0$, we have $I^a(\x)$ is closed, and hence compact. \end{lemma}

\begin{proof} For each $i\in \{1,\dots,r\}$, we define $\mathcal{N}_i:=\{n\in\N : -\a_n\le \psi_i(c_n)\le0\}$. Since $\bigcup_{i=1}^{r} \mathcal{N}_i=\N$, we have that the set $\mathcal{R}:=\{i\in\{1,\dots,r\} : |\mathcal{N}_i |=\infty\}$ is nonempty. Assume that $$\mathcal{R}=\{i_1,i_2,\dots,i_m\}\;\;\hbox{where}\;\;i_1<i_2<\cdots<i_m.$$
 For each $j\not \in\mathcal{R}$ and for each $\ell\in\{1,\dots,m\}$, we eliminate from $\mathcal{N}_{i_{\ell}}$ all the finite number of indices $n$ for which $-\a_n\le\psi_j(c_n)\le0$, if such indices exist. Thus, we now ensured that  $$\I^{\a_n}_{c_n}\subset \mathcal{R}\;\;\hbox{for all}\;\;n\in \bigcup_{\ell=1}^{m}\mathcal{N}_{i_{\ell}}. $$ 
 Now to construct the required subsequence and the set $\J_o$, we apply the following algorithm:
 \begin{enumerate}[(1)]
 \item  Let $\ell=2$ and $\J_o=\{i_1\}$.
 \item  If $\mathcal{N}_{i_{1}}\cap \mathcal{N}_{i_{\ell}}$ is an infinite set, then we replace $\mathcal{N}_{i_{1}}$ by $\mathcal{N}_{i_{1}}\cap \mathcal{N}_{i_{\ell}}$, and we add $i_2$ to $\J_o$. Otherwise, we replace $\mathcal{N}_{i_{1}}$ by   $\mathcal{N}_{i_{1}}\setminus (\mathcal{N}_{i_{1}}\cap \mathcal{N}_{i_{\ell}})$.
 \item  We increment $\ell$ by $1$. If $\ell=m+1$, then go to (4). Otherwise, go to (2).
 \item  Halt.
 \end{enumerate}
 At the end of the algorithm, the so-obtained set of indices $\mathcal{N}_{i_1}$ will be an infinite set. Moreover, if we consider $(c_{n_k})_k$ to be the subsequence of $(c_n)_n$ associated to $\mathcal{N}_{i_1}$, then we clearly have that $\I^{\a_{n_k}}_{c_{n_k}}= \J_o$ for all $k$. The continuity of $\psi_i$ for each  $i\in \{1,\dots,r\}$, and the convergence of $\a_n$ to $\a_o$ and $c_n$ to $c_o$ , yield that   $\J_o\subset \I^{\a_o}_{c_o}$. We proceed to prove the \enquote{in particular} part.  Let  $\x\colon[0,1]\f C$ be continuous and let $a\ge 0$. For $t_n \in I^a(\x)$ with $t_n\to t_o \in [0,1]$, we consider $\a_n=\a_o:= a$, $c_n:=\x(t_n)$, and $c_o:=\x(t_o)$. By the first part of this lemma, there exists $\J_o \neq \emptyset$  in $\{1,\dots,r\}$ such that, up to a subsequence, $\I_{\x(t_n)}^a= \J_o\subset \I_{\x(t_o)}^a$, implying that $t_o\in I^a(\x)$. \end{proof}
 
\begin{lemma}\label{A2.2lemma}  Assume \textnormal{(A2.1)} holds. Then assumption \textnormal{(A2.3)} is  equivalent to the existence $\bar{a}>0$, $\bar{b}\in(0,1)$ and $\bar{\rho}>0$, such that for all $t\in I^{\bar{a}}(\x)$ and  for all $j\in\I_{\x(t)}^{\bar{a}}$  we have $$\sum_{\substack{i\in\I^{\bar{a}}_{\x(t)}\\ i\not=j}}|\<\nabla \psi_i(x),\nabla\psi_j(x))\>|\leq \bar{b}\|\nabla\psi_j(x)\|^2,\;\;\forall x\in \bar{B}_{\bar{\rho}}(\x(t))\cap C.$$
\end{lemma}
\begin{proof} Using an argument by contradiction in conjunction with Lemma \ref{lemmaaux1}.\end{proof}

\begin{lemma} \label{conditionH} Let $S\subset\R^n$ be a nonempty and closed set. The assumption \textnormal{(A2.4)$(i)$} is satisfied by $S$ if one of the following conditions holds:
\begin{enumerate}[$(i)$]
\item $\bdry S$ is compact.
\item $S$ is star-shaped $($which includes convex and polyhedral sets$\sp)$.
\end{enumerate}
Moreover, in the two cases above, the radius $R_o$ can be taken to be arbitrarily large.
\end{lemma}
\begin{proof} $(i)$: It is sufficient to take $y_o$ any point in $\R^n$ and $R_o$ large enough so that $\bdry S\subset B_{R_o}(y_o)$.\vspace{0.2cm}\\
$(ii)$: Let $y_o$ be a center of $S$ and let $R_o$ be any positive number. We consider $s\in\bdry S\cap (\S_{R_o}(y_o))$. We claim that $y_o-s\not \in N_S(s)$. Indeed, if not, then there exists $\rho_s>0$ such that for $\zeta_s:=\frac{y_o-s}{\|y_o-s\|}$, we have $ B_{\rho_s}\bp\left(s+\rho_s\zeta_s\right)\subset S^c.$ Thus for $s_o:=s+\bar{\rho}_s\zeta_s$, where $0<\bar{\rho}_s<\min\{\rho_s,R_o\}$, we have $s_o\in[s,y_o]\subset S$ and $s_o\in B_{\rho_s}\bp\left(s+\rho_s\zeta_s\right)\subset S^c.$ This gives the desired contradiction. \end{proof}

\begin{lemma} \label{A2.2H2}  Assume that \textnormal{(A2.1)-(A2.2)} and  \textnormal{(A2.4)$(i)$} hold. Then the assumption \textnormal{(A2.2)} is satisfied by the family of functions $(\psi_i)_{i=1}^{r+1}$, where $$\psi_{r+1}(x):=\frac{1}{2}\left(\|x-y_o\|^2-R_o^2\right),\;\;\forall x\in\R^n.$$\end{lemma}
\begin{proof} 
If not, then there exist $(c_n)_n$ in $\R^n$ and $(\l^n_1,\dots,\l_{r+1}^n)_n\in[0,\infty)^{r+1}$ such that for each $n$ we have: 
\begin{itemize}
\item $c_n \in \hat{C}:=C\cap C_{r+1}=C\cap\bar{B}_{R_o}(y_o)$ which is a compact set.
\item $\hat{\I}^0_{c_n}:=\{i\in\{1,\dots,r+1\} : \psi_i(c_n)=0\}\not=\emptyset,$ and hence $c_n\in\bdry \hat{C}$. 
\item $\l_i^n=0$ for $i\not\in\hat{\I}^0_{c_n}$, and \begin{equation}\label{A2.21} \sum_{i\in\hat{\I}^0_{c_n}}\l_i^n=1\;\;\hbox{and}\;\; \left\|\sum_{i\in\hat{\I}^0_{c_n}}\lambda_i^n\nabla\psi_i(c_n)\right\|\leq \frac{2}{n}.\end{equation}
\end{itemize}
Since (A2.2) is satisfied, once can easily deduce that for $n\geq \big\lceil\frac{1}{\eta} \big\rceil$, we have that $\l_{r+1}^n\not=0$. This yields that $c_n\in (\S_{R_o}(y_o))$, for  $n\geq \big\lceil\frac{1}{\eta} \big\rceil$. Hence the bounded sequence $(c_n)_n$ admits a subsequence, we do not relabel, that converges to $c_o\in C\cap \S_{R_o}(y_o)\subset \bdry \hat{C}$. Applying Lemma \ref{lemmaaux1}, we deduce that $(c_n)_n$ has a subsequence, we do not relabel, that satisfies $$\exists\,\emptyset\not=\hat{\I}_o\subset \{1,\dots,r+1\}\;\,\hbox{such that}\;\,\hat{\I}^0_{c_n}=\hat{ \I}_o\subset \hat{\I}^0_{c_o}\;\,\hbox{for all}\;\,n\in\N.$$
Hence, using \eqref{A2.21}, we get that  \begin{equation*} \sum_{i\in\hat{\I}_{o}}\l_i^n=1\;\;\hbox{and}\;\; \left\|\sum_{i\in\hat{\I}_{o}}\lambda_i^n\nabla\psi_i(c_n)\right\|\leq \frac{2}{n}.\end{equation*}
Now taking $n\f\infty$ in this latter and using the boundedness of the sequence $(\l^n_1,\dots,\l_{r+1}^n)_n$, we deduce the existence of $(\l^o_1,\dots,\l_{r+1}^o)\in[0,\infty)^{r+1}$ such that \begin{equation}\label{A2.22} \sum_{i\in\hat{\I}_{o}}\l_i^o=1\;\;\hbox{and}\;\; \sum_{i\in\hat{\I}_{o}}\lambda_i^o\nabla\psi_i(c_o)=0.\end{equation}
{\underline{Case 1}}: $c_o\in\inte C$.\vspace{0.1cm}\\
Then, for $n$ sufficiently large, we have $\hat{\I}_o=\hat{\I}^0_{c_n}=\{r+1\}$. This yields, using \eqref{A2.22}, that $\l_{r+1}^o=1$ and $c_o-y_o=0,$ which contradicts $c_o\in \S_{R_o}(y_o)$.\vspace{0.1cm}\\
{\underline{Case 2}}: $c_o\in\bdry C.$\vspace{0.1cm}\\
We claim that $\l_{r+1}^o\not=0$. Indeed, if not then from \eqref{A2.22} and for $\l_i^o:=0$ for all $i\in \hat{\I}^0_{c_o}\setminus\hat{\I}_o$,  we have 
$$\sum_{i\in\hat{\I}^0_{c_o}}\l_i^o=1\;\;\hbox{and}\;\; \sum_{i\in\hat{\I}^0_{c_o}}\lambda_i^o\nabla\psi_i(c_o)=0,$$
which contradicts (A2.2). Hence, $\l_{r+1}^o\not=0$, and this gives using \eqref{A2.22}, that $$y_o-c_o=\sum_{\substack{i\in\hat\I_{o} \\ i\not=r+1}}\frac{\l_i^o}{\l_{r+1}}\nabla\psi_i(c_o)=\sum_{i\in\hat{\I}^0_{c_o}}\frac{\l_i^o}{\l_{r+1}}\nabla\psi_i(c_o)\in N_C^P(c_o),$$
where $\l_i^o:=0$ for $i\in \hat{\I}^0_{c_o}\setminus\hat{\I}_o$. This contradicts (A2.4)$(i)$.  \end{proof}

\subsection{Proofs of \textnormal{Propositions \ref{proppsigk}} and \textnormal{\ref{propcgk(k)}}} 

\begin{proofb} {\bf of Proposition \ref{proppsigk}}. $(i)$: The definition of $\psi_{\gk}$ in \eqref{psigkdef} and assumption (A2.1)  yield that, for all $k\in \N$,  $\psi_{\gk}$ is $\CO^{1,1}$ on $C+\rho B$. The bound in \eqref{gradbdd} follows from \eqref{gradpsigk} and the definition of $\bar{M}_{\psi}$.  For the remaining properties,  see \cite[Subsection 2.2]{Fang}.\vspace{0.2cm}\\
$(ii)$: If this statement is not true, there exist $({\gamma}_{k_n})_n$  with $k_n\ge n$,   $x_{{\gamma}_{k_n}}\in \R^n$ with  $\psi_{{\gamma}_{k_n}}(x_{{\gamma}_{k_n}})=0$, and $ z_{{\gamma}_{k_n}}\in B_{\frac{2}{n}}(x_{{\gamma}_{k_n}})$ such that for all $n> \frac{2}{\rho}$ we have
$$\|\nabla\psi_{{\gamma}_{k_n}}(z_{{\gamma}_{k_n}})\|\leq 2\eta.$$ 
Using  \eqref{pisgkineq} and the definition of $\psi_{{\gamma}_{k_n}}$, it follows  that \begin{equation}\label{twopsi1}  \psi(x_{{\gamma}_{k_n}})\leq 0\leq \psi(x_{{\gamma}_{k_n}}) +\frac{\ln(r)}{\gk}\;\,\hbox{and }\;\, \frac{\left\| \sum_{i=1}^{r}e^{{\gamma}_{k_n} \psi_i(z_{{\gamma}_{k_n}})}\nabla\psi_i(z_{{\gamma}_{k_n}}) \right\|}{\sum_{i=1}^{r}e^{{\gamma}_{k_n} \psi_i(z_{{\gamma}_{k_n}})}}  \leq 2\eta.\end{equation}
Then, $x_{{\gamma}_{k_n}}\in C$ and hence,  there exists a subsequence, we do not relabel, of $({\gamma}_{k_n})_n$ along which  $(x_{{\gamma}_{k_n}})_n$ and   $(z_{{\gamma}_{k_n}})_n$  converge to the same element $z_o\in C.$ Taking $n\f\infty$ in \eqref{twopsi1} and using the fact that $e^{{\gamma}_{k_n}\psi_i(z_{{\gamma}_{k_n}})}\f0$ whenever $\psi_i(z_o)<0$, we get the existence of a sequence of nonnegative numbers $(\l_i)_{i\in\I^0_{z_o}}$ such that
$$\psi(z_o)=0 ,\;\;\left\| \sum_{i\in\I^0_{z_o}}\l_i\nabla\psi_i(z_o) \right\|\leq 2\eta\;\;\hbox{and}\;\;\sum_{i\in\I^0_{z_o}}\l_i=1.$$
This contradicts  assumption (A2.2) since $\psi(z_o)=0$ is equivalent to  $\I^0_{z_o}\not=\emptyset$.\vspace{0.2cm}\\
$(iii)$: If this statement is not true,  there exist $({\gamma}_{k_n})_n$  with $k_n\geq n$ and  $x_{{\gamma}_{k_n}}\in C^{\gkn}\subset C$  such that \begin{equation}\label{epsi1} \|\nabla\psi_{\gkn}(x_{\gkn})\|\leq \eta\;\;\hbox{and}\;\,-\frac{1}{n}\leq \psi_{\gkn}(x_{\gkn})\leq0.\end{equation}
This yields that $x_{\gkn}\in C$ and $\psi(x_{\gkn})\f0$. Since $C$ is compact, we can assume that $x_{\gkn}\f z_o\in C$, and hence $\psi(z_o)=\lim_{n\f 0} \psi(x_{\gkn})=0$. Now, from \eqref{epsi1}$(a)$ and \eqref{subpsi}, we have 
$$\frac{\|\sum_{i=1}^{r}e^{{\gamma}_{k_n} \psi_i(x_{{\gamma}_{k_n}})}\nabla\psi_i(x_{{\gamma}_{k_n}})\|}{\sum_{i=1}^{r}e^{{\gamma}_{k_n} \psi_i(x_{{\gamma}_{k_n}})}}\leq \eta.$$
Taking $n\f \infty$ in this latter and using that $e^{{\gamma}_{k_n}\psi_i(x_{{\gamma}_{k_n}})}\f0$ whenever $\psi_i(z_o)<0$, we get the existence of a sequence of nonnegative numbers $(\l_i)_{i\in\I^0_{z_o}}$ such that $$\left\| \sum_{i\in\I^0_{z_o}}\l_i\nabla\psi_i(z_o) \right\|\leq \eta\;\;\hbox{and}\;\;\sum_{i\in\I^0_{z_o}}\l_i=1.$$
This contradicts  assumption (A2.2), since $\psi(z_o)=0$ implies that $\I^0_{z_o}\not=\emptyset$.
\end{proofb}

\begin{proofb} {\bf of Proposition \ref{propcgk(k)}}. $(i)$: By Proposition \ref{proppsigk}, we have that for each $k\geq k_3$, $\psi_{\gk}$ satisfies the {\it same assumptions} satisfied by the function $\psi$ of \cite[Theorem 3.1]{VCpaper}. Hence, $(i)$ follows immediately from \cite[Theorem 3.1$(i)$]{VCpaper} where $\psi$ is replaced by $\psi_{\gk}$. \vspace{0.2cm}\\
$(ii)$: First we prove that \begin{equation}\label{inteCsubsetinteCgk(k)} \inte C\subset \bigcup_{k\in \N} \inte C^{\gk}(k).\end{equation} Let $x\in\inte C$. Since $\psi(x)<0$ and $\big(\a_k+\frac{\ln r}{\gk}\big)\f 0$, there exists $k_x\geq k_4$ such that $\a_{k_x}+\frac{\ln r}{\gamma_{k_x}}<-\psi(x).$
This yields, using \eqref{pisgkineq}, that $\psi_{\gamma_{k_x}}(x)<-\a_{k_x}$, and hence $$x\in \inte C^{\gamma_{k_x}}(k_x)\subset \bigcup_{k\in \N} \inte C^{\gk}(k).$$ 
This terminates the proof of \eqref{inteCsubsetinteCgk(k)}. Hence, $$\inte C\subset \bigcup_{k\in \N} \inte C^{\gk}(k)\subset \bigcup_{k\in \N} C^{\gk}(k)\subset  \bigcup_{k\in \N}  \inte C^{\gk}\subset \inte C.$$
Therefore, the equation \eqref{union.Cgk(k)} holds true. Now, since $(\psi_{\gk})_k$ is monotonically nonincreasing in terms of $k$, and $(\a_k)_k$ is a decreasing sequence, we conclude that the sequence $(C^{\gk}(k))_k$ is a nondecreasing sequence. This gives that the Painlev\'e-Kuratowski limit of the sequence $(C^{\gk}(k))_k$ satisfies 
\begin{equation}\label{limitCgk(k)} \lim_{k\f\infty} C^{\gk}(k)=\clo\Bigg(\bigcup_{k\in \N} C^{\gk}(k)\Bigg). \end{equation} 
Upon taking the closure of $\inte C$ in \eqref{union.Cgk(k)} and using from Proposition \ref{prop1}$(i)$ that $C=\clo({\inte C})$, equation \eqref{limitCgk(k)} yields that the Painlev\'e-Kuratowski limit of the sequence $(C^{\gk}(k))_k$ is $C$.\vspace{0.2cm}\\
$(iii)$: For $c\in\bdry C$,  we set $-2\a_c:=\max\{\psi_i(c) : i\not\in\I^0_c\}$, whenever $\I^0_c \subsetneq \{1,\dots, r\}$. We have $\a_c>0$ and  $\psi_i(c)\leq -2\a_c$ for all $i\not\in\I^0_c$. The continuity of $\psi_i$ yields the existence of $\r_1>0$ such that $\psi_i(x)<-\a_c$, for all $x\in \bar{B}_{2\r_1}(c)$ and for all $i\not\in\I^0_c$. Let  $d_c$ be the nonzero vector of Lemma \ref{dcexistence}.  Choose $\k_1\geq k_3$ large enough so that for all $k\ge \k_1$ we have $\frac{\ln(r)}{\gk}+\a_k\leq \a_c$ and  $\sigma_k\leq \r_1$, where $\a_k$ and  $\sigma_k$  are the sequences defined in \eqref{sigmadef}. Then, for all $k\geq\k_1$, and for all $x\in \bar{B}_{\r_1}(c)$ we have:\begin{itemize}
\item $\left\|x+\sigma_k\frac{d_c}{\|d_c\|}-c\right\|\leq \|x-c\|+\sigma_k\leq 2\r_1$, 
\item $\psi_i\left(x+\sigma_k\frac{d_c}{\|d_c\|}\right)< -\a_c\leq -\frac{\ln(r)}{\gk}-\a_k$ for all $i\not\in\I^0_c$.
\end{itemize}  
Therefore, \begin{equation}\label{inotinIc}\psi_i\left(x+\sigma_k\frac{d_c}{\|d_c\|}\right)+\frac{\ln(r)}{\gk}<-\a_k,\;\;\forall i\not\in \I^0_{c},\;\ \forall k\geq\k_1,\,\hbox{and}\;\forall x\in \bar{B}_{\bar{r}_1}(c).\end{equation}
Now, for $i\in \I^0_c$, by using  Lemma \ref{dcexistence} we get $$ \left\<\frac{d_c}{\|d_c\|},\nabla\psi_i(c)\right\>\le -\frac{4\eta^2}{r\bar{M}_{\psi}},\;\;\forall i\in\I^0_c.$$ 
The continuity of $\nabla\psi_i$ yields the existence of $0<r_{\c}\leq \r_1$ such that   \begin{equation} \label{localversion} \left\<\frac{d_c}{\|d_c\|},\nabla\psi_i(x)\right\><- \frac{2\eta^2}{r\bar{M}_{\psi}},\;\;\forall i\in\I^0_c\;\,\hbox{and}\,\;\forall x\in \bar{B}_{2r_{\c}}(c).\end{equation} We consider $k_c\geq \k_1$ large enough so that $\sigma_k\leq r_{\c}$. Then, for all $k\geq k_c$ and for all $x\in  \bar{B}_{r_{\c}}(c)$, we have $$\left\|x+\sigma_k\frac{d_c}{\|d_c\|}-c\right\|\leq \|x-c\|+\sigma_k\leq 2r_{\c}. $$
Hence by the mean value theorem applied to $\psi_i$ on $[x+\sigma_k\frac{d_c}{\|d_c\|},x]$, for $i\in\I^0_c$, for $k\geq k_c$, and for $x\in  \bar{B}_{r_{\c}}(c)\cap C$, and using \eqref{localversion} and \eqref{sigmadef}, we obtain the existence of $z\in [x+\sigma_k\frac{d_c}{\|d_c\|},x] \subset  \bar{B}_{2r_{\c}}(c)$  such that $$\psi_i\left(x+\sigma_k\frac{d_c}{\|d_c\|}\right)=\psi_i(x) + \sigma_k\left\<\nabla\psi_i(z),\frac{d_c}{\|d_c\|}\right\> < 0 -\frac{\ln(r)}{\gk}-\a_k.$$
Combining this latter with \eqref{inotinIc}, we conclude that \begin{equation*}\psi_i\left(x+\sigma_k\frac{d_c}{\|d_c\|}\right)+\frac{\ln(r)}{\gk}<-\a_k,\;\;\forall i\in\{1,\dots,r\},\;\;\forall k\geq k_c,\;\hbox{and}\;\forall x\in \bar{B}_{r_{\c}}(c)\cap C.\end{equation*}
Therefore, using \eqref{psidef} and \eqref{pisgkineq}, we deduce that \begin{equation*}\psi_{\gk}\left(x+\sigma_k\frac{d_c}{\|d_c\|}\right)<-\a_k,\;\;\forall k\geq k_c\;\,\hbox{and}\;\,\forall x\in \bar{B}_{r_{\c}}(c)\cap C.\end{equation*}
The proof of $(iii)$ is terminated. \end{proofb}


\begin{thebibliography}{00}

\bibitem{outrata} L. Adam, J. Outrata, On optimal control of a sweeping process coupled with an ordinary differential equation, {\it Discrete Contin. Dyn. Syst. B} 19(9) (2014), 2709--2738.

\bibitem{adly} S. Adly, F. Nacry, L. Thibault, Discontinuous sweeping process with prox-regular sets, ESAIM: COCV, 23:4 (2017), 1293--1329.

\bibitem{brokate} M. Brokate, P. Krej\v{c}\'{i}, Optimal control of ODE systems involving a rate independent variational inequality, Discrete and continuous dynamical systems series B 18 (2013), 331--348.

\bibitem{brudnyi} A. Brudnyi, Y. Brudnyi, Methods of Geometric Analysis in Extension and Trace Problems, Volume 1, Monographs in Mathematics, 102. Birkh\"{a}user/Springer Basel AG, Basel, 2012.

\bibitem{ccmn} T.H. Cao, G. Colombo, B. Mordukhovich, D. Nguyen, Optimization of fully controlled sweeping processes, J. Differ. Equ. 295 (2021), 138--186.

\bibitem{ccmnbis} T.H. Cao, G. Colombo, B. Mordukhovich, D. Nguyen, Optimization and discrete approximation of sweeping processes with controlled moving sets and perturbations, J. Differ. Equ. 274 (2021), 461--509.

\bibitem{cmo0} T.H. Cao, B. Mordukhovich, Optimal control of a perturbed sweeping process via discrete approximations, Discrete Contin. Dyn. Syst. Ser. B 21, (2016), 3331--3358.

\bibitem{cmo} T.H. Cao, B. Mordukhovich, Optimality conditions for a controlled sweeping process with applications to the crowd motion model, Disc. Cont. Dyn. Syst. Ser. B 22 (2017), 267--306.

\bibitem{cmo2} T.H. Cao, B. Mordukhovich, Optimal control of a nonconvex perturbed sweeping process, J.Differ. Equ. 266 (2019), 1003--1050.

\bibitem{clarkeold} F.H. Clarke, {\it Optimization and Nonsmooth Analysis}, Wiley Interscience, New York, 1983, Republished as Vol. 5 of Classics in Applied Mathematics, S.I.A.M., Philadelphia, 1990.

\bibitem{clarkebook} F.H. Clarke, Functional Analysis, Calculus of Variations and Optimal Control, Graduate Texts in Mathematics, 264. Springer, London, 2013.

\bibitem{clsw} F.H. Clarke, Yu. Ledyaev, R.J. Stern, P.R. Wolenski, Nonsmooth Analysis and Control Theory, Graduate Texts in Mathematics, 178, Springer-Verlag, New York, 1998.

\bibitem{prox} F.H. Clarke, R. Stern, P. Wolenski, { Proximal smoothness and the lower-$C^2$ property}, J. Convex Analysis 2 (1995), 117--144.

\bibitem{chhm2} G. Colombo, R. Henrion, N.D. Hoang, B.S. Mordukhovich, Optimal control of the sweeping process, Dyn. Contin. Discrete Impuls. Syst. Ser. B 19 (2012), 117--159.

\bibitem{chhm} G. Colombo,  R. Henrion, N.D. Hoang, B.S. Mordukhovich, Optimal control of the sweeping process over polyhedral controlled sets, J. Differ. Equ. 260 (2016) no. 4, 3397--3447. 

\bibitem{cmn0} G. Colombo, B. Mordukhovich, D. Nguyen, Optimization of a perturbed sweeping process by constrained discontinuous controls, SIAM Journal on Control and Optimization, 58 (2020) no. 4, 2678--2709.

\bibitem{pinho} M.d.R. de Pinho, M.M.A. Ferreira, G.V. Smirnov, Optimal Control Involving Sweeping Processes, Set-Valued Var. Anal. 27 (2019), no. 2, 523--548.

\bibitem{pinhoEr} M.d.R. de Pinho, M.M.A. Ferreira, G.V. Smirnov, Correction to: Optimal Control Involving Sweeping Processes, Set-Valued Var. Anal. 27 (2019), 1025--1027.

\bibitem{pinhonum} M.d.R. de Pinho, M.M.A. Ferreira, G.V. Smirnov, Optimal Control with Sweeping Processes: Numerical Method, J Optim Theory Appl 185, (2020), 845--858. 

\bibitem{pinhonew} M.d.R. de Pinho, M.M.A. Ferreira, G.V. Smirnov, Necessary conditions for optimal control problems with sweeping systems and end point constraints, Optimization, 71:11, (2021), 3363--3381.

\bibitem{pinho22} M.d.R. de Pinho, M.M.A. Ferreira, G.V. Smirnov, A Maximum Principle for optimal control problems involving sweeping processes with a nonsmooth set, arXiv:2301.13620v1 [math.OC].

\bibitem{palladino} C. Hermosilla,  M. Palladino, Optimal Control of the Sweeping Process with a Nonsmooth Moving Set, SIAM Journal on Control and Optimization, 60:5 (2022), 2811--2834.

\bibitem{Fang} X.-S. Li, S.-C. Fang, On the entropic regularization method for solving min-max problems with applications. Mathematical Methods of Operations Research 46 (1997), 119--130. 

\bibitem{mordubook} B.S. Mordukhovich, Variational Analysis and Generalized Differentiation, I: Basic Theory, Springer, Berlin, 2006.

\bibitem{moreaudecomp} J.J. Moreau, D\'ecomposition orthogonale d'un espace hilbertien selon deux c\^ones mutuellement polaires, Comptes rendus hebdomadaires des s\'eances de l'Acad\'emie des sciences, Gauthier-Villars, 255 (1962), pp.238--240. 

\bibitem{moreau1} J.J. Moreau, Rafle par un convexe variable, I, Trav. Semin. d'Anal. Convexe, Montpellier 1, Expos\'e 15 (1971) 36 pp.

\bibitem{moreau2} J.J. Moreau, Rafle par un convexe variable, II, Trav. Semin. d'Anal. Convexe, Montpellier 2, Expos\'e 3 (1972) 43 pp.

\bibitem{moreau3} J.J. Moreau, Evolution problem associated with a moving convex set in a Hilbert space, J. Differ. Equations 26 (1977), 347--374.

\bibitem{verachadinum} C. Nour, V. Zeidan, Numerical solution for a controlled nonconvex sweeping process, IEEE Control Systems Letters, vol. 6, (2022), 1190--1195.

\bibitem{VCpaper} C. Nour, V. Zeidan, Optimal control of nonconvex sweeping processes with separable endpoints: Nonsmooth maximum principle for local minimizers, J. Differ. Equ. 318 (2022), 113-168.

\bibitem{prt} R.A. Poliquin, R.T. Rockafellar, L. Thibault, Local differentiability of distance functions, Trans. Amer. Math. Soc. 352 (2000), 5231--5249.

\bibitem{rockwet} R.T. Rockafellar, R.J.-B. Wets, Variational analysis, Grundlehren der Mathematischen Wissenschaften, 317, Springer-Verlag, Berlin, 1998.

\bibitem{show} R. E. Showalter, Monotone Operators in Banach Space and Nonlinear Partial Differential Equations, AMS, Mathematical Surveys and Monographs, Volume 49, 1997.

\bibitem{vinter} R.B. Vinter, Optimal Control, Birkh\"{a}user, Systems and Control: Foundations and Applications, Boston, 2000.

\bibitem{verachadi} V. Zeidan, C. Nour, H. Saoud, A nonsmooth maximum principle for a controlled nonconvex sweeping process, J. Differ. Equ. 269 (2020), 9531--9582.

\end{thebibliography}
\end{document}